\documentclass{amsart}
\usepackage{amsmath, amssymb, epsfig, graphicx, color }
\usepackage[all]{xy}

\newtheorem{thm}{Theorem}[section]
\newtheorem{lem}[thm]{Lemma}

\newtheorem{prop}[thm]{Proposition}

\newtheorem*{thmnonum}{Theorem}

\newcommand{\integer}{\mathbb{Z}}
\newcommand{\nat}{\mathbb{N}}

\newcommand{\cantor}{\{0,1\}^{\nat}}
\newcommand{\two}{\{0,1\}}
\newcommand{\homcantor}{{\mathcal H} (\cantor)}
\newcommand{\contcantor}{{\mathcal C} (\cantor)}

\newcommand{\bn}{{\mathcal B}_n}
\newcommand{\p}{\mathcal P}
\newcommand{\q}{\mathcal Q}
\newcommand{\cC}{\mathcal C}

\newcommand{\ov}{\overline}
\newcommand{\edge}[1]{\ensuremath{\overrightarrow  {#1 }}}

\newcommand{\gr}{\operatorname{Gr}}
\newcommand{\mesh}{\operatorname{mesh}}
\newcommand{\diam}{\operatorname{diam}}

\title[Graph Theoretic Structure of Maps of the Cantor Space]
{Graph Theoretic Structure of Maps of the Cantor Space}
\author[Bernardes]{Nilson C. Bernardes Jr.}
\author[Darji]{Udayan B. Darji}
\address[Bernardes]{Instituto de Matem\'atica, Universidade Federal do Rio de Janeiro,
Caixa Postal 68530, Rio de Janeiro, RJ 21945-970, Brasil} 
\email[Bernardes] {bernardes@im.ufrj.br}
\address[Darji]{Department of Mathematics, University of Louisville, Louisville, KY 40292, USA}
\email[Darji]{ubdarj01@louisville.edu}
\thanks{}
\begin{document}
\maketitle


\begin{abstract} In this paper we develop unifying graph theoretic
techniques to study the dynamics and the structure of spaces
$\homcantor$ and $\contcantor$, the space of homeomorphisms and
the space of self-maps of the Cantor space, respectively.
Using our methods, we give characterizations which determine when two
homeomorphisms of the Cantor space are conjugate to each other.
We also give a new characterization of the comeager conjugacy class of the
space $\homcantor$. The existence of this class was established by
Kechris and Rosendal and a specific element of this class was described
concretely by Akin, Glasner and Weiss. Our characterization readily implies
many old and new dynamical properties of elements of this class.
For example, we show that no element of this class has a Li-Yorke pair,
implying the well known Glasner-Weiss result that there is a comeager subset
of $\homcantor$ each element of which has topological entropy zero.
Our analogous investigation in $\contcantor$ yields a surprising result:
there is a comeager subset of $\contcantor$ such that any two elements of
this set are conjugate to each other by an element of $\homcantor$.
Our description of this class also yields many old and new results
concerning dynamics of a comeager subset of $\contcantor$. 
\end{abstract}


\section{Introduction}
In recent years $\homcantor$, the group of homeomorphisms of the Cantor space,
has enjoyed attention from mathematicians working in diverse fields such as
dynamical systems \cite{AHK}, \cite{hochman} and model theory \cite{KR}.
From the dynamics point of view, $\cantor$ is the symbol space on two
variables and a fundamental tool in analyzing topological dynamics of
complicated systems. From the model theory point of view, $\homcantor$
is important as it is isomorphic to the group of automorphisms of the
countable, atomless boolean algebra. The study of $\homcantor$ from the
dynamical systems approach utilizes analytic techniques while the model
theory viewpoint exploits fact that $\cantor$ can be viewed as the
Fra\"iss\`e limit of finite Boolean algebras. Our approach is fundamentally
different in that to each $h \in \homcantor$ and each finite partition $\p$
of clopen subsets of $\cantor$, we associate a digraph $\gr(h,\p)$.
We use geometric and graph theoretic properties of these digraphs to deduce
dynamical properties of $h$.

In one of our main results, we establish characterizations of when two
elements of $\homcantor$ are topologically conjugate to each other.
In addition to its importance from the algebraic point of view,
this result is also very important from the dynamical systems point of view.
The reason for this is that two topologically conjugate homeomorphisms have
the same topological dynamics. Hence, dynamical properties of every element
of a conjugacy class can be described by the dynamical properties of a
single member of that class. Characterizations of the conjugacy relation
for other groups such as the symmetric group on $\nat$,
the group of automorphisms of the random graph and the group of order
preserving automorphisms of the rationals were described by Truss in
\cite{truss}.

Truss \cite{truss} also defined the concept of generics in a Polish group.
We say that a {\em Polish group $G$ admits generics} if $G$ has a comeager
conjugacy class. In this case, we say that {\em a generic element
of $G$ has Property P} if every element of the comeager conjugacy class has
Property P. We caution the reader that in dynamical systems, topology and
other fields, the word generic is often used in a different manner. Namely,
it refers to a comeager subset of a complete metric space with no reference
to algebraic structure. We will use the phrase ``comeager conjugacy class''
instead of the word ``generic'' to avoid confusion. 

In 2001 Glasner and Weiss \cite{gw} showed that $\homcantor$ has a dense
conjugacy class. They called the property of having a dense conjugacy class
the topological Rohlin property. This was also shown by Akin, Hurley
and Kennedy in their monograph \cite{AHK}. In the same monograph they 
raised the question whether $\homcantor$ has a
comeager conjugacy class. This was settled affirmatively in 2007 by
Kechris and Rosendal in \cite{KR} using model theoretic techniques.
Akin, Glasner and Weiss \cite{agw}, in 2008, gave a concrete construction
of what they called a ``Special Homeomorphism'' whose conjugacy class is
comeager. As an application of the techniques developed in the present
article, we give a geometric/graph theoretic description of the elements of
this comeager conjugacy class. We describe these homeomorphisms below to give
a flavor of the ideas developed here.

Let $h \in \homcantor$ and $\p$ be a finite partition of clopen subsets of
$\cantor$. Then, $\gr(h, \p)$ is a digraph whose vertex set is $\p$ and
$\edge{ab}$ is an edge of $\gr(h,\p)$ if and only if
$h(a) \cap b \neq \emptyset$.
A digraph $H$ is a {\em loop} if the vertex set of $H$ is
$\{v_1, \ldots, v_{n}\}$ and the edges of $H$ are
$\edge{v_{n}v_1}$ and $\edge {v_iv_{i+1}}$ for $1 \le i < n$.
In this case, we say that $H$ is a loop of {\em length} $n$. A digraph
$H$ is a {\em dumbbell} if the vertex set of $H$ is the union of three
disjoint sets $l_1=\{u_1, \ldots, u_r\}$, $p=\{v_1, \ldots, v_s\}$,
$l_2=\{w_1, \ldots, w_t\}$ and the edges of $H$ are
\begin{itemize}
\item the edges of the loops formed by $l_1$ and $l_2$,
\item the edges of the path $p$, i.e., $\edge {v_i v_{i+1}}$
      for $1 \le i < s$, and
\item $\edge {u_1 v_1}$, $\edge {v_sw_1}$. 
\end{itemize} 
In this case we say that $H$ is a dumbbell of {\em type $(r,s,t)$}.
If $r=t$, then we say that the dumbbell is {\em balanced with plate weight
$r$}. Assume that a component $H$ of $\gr(h,\p)$ is a dumbbell and let us
denote $H$ as above. We say that {\em $H$ contains a left loop of $h$}
(resp.\ {\em a right loop of $h$}) if there is a nonempty clopen subset
$a$ of $u_1$ (resp.\ of $w_1$) such that $h^r(a) = a$ (resp.\ $h^t(a) = a$).

Now we are ready to describe the comeager conjugacy class of $\homcantor$. 

\begin{thmnonum}
The set of all $h \in \homcantor$ with the following property is a comeager
conjugacy class of $\homcantor$: \\
For every $m \in \nat$, there are a partition $\p$ of $\cantor$
of mesh $< 1/m$ and a multiple $q \in \nat$ of $m$ such that
every component of $\gr(h,\p)$ is a balanced dumbbell with plate weight $q!$
that contains both a left and a right loop of $h$.
\end{thmnonum}

We point out that using projective Fra\"iss\`e limits, Kwiatkowska \cite{kwa}
has shown that $\homcantor$ has ample generics, a property stronger
than having a comeager conjugacy class. We would also like to point out that
the ideas used by Akin, Hurley and Kennedy in \cite{AHK} to prove that
$\homcantor$ has a dense conjugacy class has some distant resemblance to our 
techniques.

Using our graph theoretic techniques, we prove a surprising result: there is 
a comeager subset of $\contcantor$ such that any two elements of this set
are conjugate to each other by an element of $\homcantor$.  This is done
by giving a geometric/graph theoretic description of this class in a manner
similar to that of  $\homcantor$.

The notion of chaos is another well studied concept in topological dynamics. 
There are several different notions of chaos. For instance, page 1306 of
\cite{aghsy} contains a table with 11 notions of chaos and the relationships
among them, including the 4 best known notions, namely: positive topological
entropy, chaos in the sense of Devaney, weak mixing and chaos in the sense of
Li and Yorke. As is well-known, Li-Yorke chaos is the weakest of all notions
of chaos \cite{bgkm}. In \cite{gw} Glasner and Weiss showed that a comeager
subset of $\homcantor$ has topological entropy zero. Hence in some sense an
element chosen at ``random'' from $\homcantor$ has topological entropy zero
and therefore is not chaotic in this sense. As a simple corollary to our
investigation, we show that homeomorphisms of the comeager conjugacy class
have a much stronger property, namely that they have no Li-Yorke pair.

We show that homeomorphisms of this comeager conjugacy class also
have other properties which make them tame. In this direction we show that
each such homeomorphism has the shadowing property and the restriction
of the homeomorphism to each of its $\omega$-limit sets is topologically
conjugate to the universal odometer. Hochman \cite{hochman} 
showed that among all the transitive homeomorphisms of the Cantor space,
the set of homeomorphisms topologically conjugate to the universal odometer
is comeager. We also show that for all $h$ in the comeager conjugacy class,  
the set of recurrent points of $h$ is equal to the set of chain recurrent
points of $h$. Moreover, $h$ is chain continuous on a dense open subset
of $\cantor$ but not equicontinuous on an uncountable set.

We also show analogous dynamical properties of a comeager subset of
$\contcantor$. In particular, we show that the set of all $f \in \contcantor$
which have the following properties is comeager in $\contcantor$:
\begin{itemize}
\item $f$ has no Li-Yorke pair.
\item The restriction of $f$ to each of its $\omega$-limit sets is
      topologically conjugate to the universal odometer.
\item $f$ is chain continuous at every point.
\end{itemize}
Earlier it was shown in \cite{dd} that the set of $f \in \contcantor$ with
topological entropy zero and no periodic point is comeager in $\contcantor$.
In \cite{dds} it was shown that there is a comeager set of $f \in \contcantor$
such that for a comeager set of $\sigma \in \cantor$, the restriction of $f$
to the $\omega$-limit set of $f$ at $\sigma$ is topologically conjugate to
the universal odometer.

This paper is organized as follows:
Section~2 develops properties of $\gr(f,\p)$ for $f \in \contcantor$
and for $f \in \homcantor$, and contains an important approximation theorem,
Theorem~\ref{approxthm}.
In Section~3 we give characterizations of when two homeomorphisms
(or two continuous maps) of the Cantor space are topologically conjugate
to each other.
Section~4 contains an useful geometric/graph theoretic description of the
homeomorphisms of the Cantor space whose conjugacy class is comeager.
Moreover, by applying this description we obtain several results concerning
dynamical properties of the comeager conjugacy class of $\homcantor$.
In Section~5 we prove the surprising result that there is a comeager subset
of $\contcantor$ such that any two elements of this set are conjugate
to each other by an element of $\homcantor$.
Using the description of this set,
we prove dynamical properties of elements of this set.


\section{Approximation Theorem}
By Cantor space we mean any compact, $0$-dimensional metric space without
isolated points. The principal model of the Cantor space we use is $\cantor$
endowed with the product topology, where $\two$ is given the discrete topology.
This topology is generated by the metric $d(\sigma,\tau) = \frac{1}{n}$
where $n$ is the least positive integer where $\sigma(n) \neq \tau(n)$
if such an integer exists and $d(\sigma,\tau) =0$, otherwise.

By a partition of $\cantor$ we mean a finite collection of nonempty pairwise
disjoint clopen sets whose union is $\cantor$. A map $\nu : \p \to \q$,
between partitions $\p$ and $\q$ of $\cantor$, is called a refinement map if
$$
\nu(a) \supset a \ \ \text{ for all } a \in \p.
$$
In this case, we say that $\p$ is a refinement of $\q$. Note that a
refinement map is necessarily surjective.

For each finite collection $\cC$ of nonempty subsets of $\cantor$,
we define the mesh of $\cC$ by
$$
\mesh(\cC) = \max_{A \in \cC} \diam(A).
$$

If $\sigma$ is a finite string of $0$'s and $1$'s, then $[\sigma]$ denotes
the set of all points of $\cantor$ which are extensions of $\sigma$.
For each $n \in \nat$, we consider the partition
$$
\bn= \{ [\sigma] : \sigma \in \two^n\}
$$
of $\cantor$. Note that
$$
\mesh(\bn) \to 0 \ \ \text{ as } n \to \infty.
$$
Moreover, $\bigcup \bn$ is a basis of clopen sets for the topology of the
Cantor space.

We use $\contcantor$ (resp.\ $\homcantor$) to denote the space of
all continuous maps (resp.\ of all homeomorphisms) of the Cantor space,
endowed with the following metric:
$$
\tilde{d}(f,g) = \max _{\sigma \in \cantor} d(f(\sigma), g(\sigma)).
$$
If $f, g \in \contcantor$ and $\p$ is a partition of $\cantor$,
then $f \sim_\p g$ means that $f(\sigma)$ and $g(\sigma)$ lie in the same
member of $\p$ for every $\sigma \in \cantor$. Note that
$$
f \sim_\p g \ \ \Longrightarrow \ \ \tilde{d}(f,g) \leq \mesh(\p).
$$

Central to our investigation is the notion of a digraph ($=$ directed graph).
A digraph $G$ consists of a finite set $V(G)$ of vertices together with
a set $E(G)$ of directed edges between vertices.
By a {\em left end} of $G$ (resp.\ a {\em right end} of $G$) we mean a
vertex $v$ of $G$ that has no incoming edge (resp.\ no outgoing edge).
We say that $G$ is a {\em digraph without right ends}
(resp.\ a {\em digraph without ends}) if $G$ has no right end
(resp.\ no left end and no right end).
If $G$ and $H$ are digraphs, then a digraph map $\phi : H \to G$ is a
map from the vertex set of $H$ into the vertex set of $G$ such that
$\edge{\phi(u)\phi(v)}$ is an edge of $G$ whenever $\edge{uv}$ is an
edge of $H$. We say that $\phi$ is surjective when it is a surjection
between the sets of vertices (but it need not be surjective on the sets
of edges).
By a component of a digraph $G$ we simply mean a largest
(in vertices and edges) subgraph $H$ of $G$ such that
given any two vertices $a, b$ in $H$, there are vertices
$a_1, \ldots, a_n$ in $H$ such that $a_1 =a, a_n =b$ and for any
$1 \le i < n$, $\edge{a_ia_{i+1}}$ or $\edge{a_{i+1}a_i}$ is an edge of $H$.

To each $f \in \contcantor$ and each partition $\p$ of $\cantor$,
we associate a digraph $\gr(f,\p)$ in the following fashion: the vertices
of $\gr(f,\p)$ are the elements of $\p$ and for sets $a, b \in \p$,
directed edge $\edge{ab} \in \gr(f,\p)$ if and only if
$f(a) \cap b \neq \emptyset$. Note that $\gr(f,\p)$ is always a
digraph without right ends. If $f$ is surjective, then $\gr(f,\p)$ is a
digraph without ends. If $\nu : \p \to \q$ is a refinement map, then
$\nu : \gr(f,\p) \to \gr(f,\q)$ is a surjective graph map (in the present
case, it is surjective on edges too). Technically speaking,
$$
\p \to \gr(f,\p)
$$
is a functor from the category of partitions of the Cantor space
(whose morphisms are the refinement maps) to the category of digraphs
(whose morphisms are the digraph maps). Moreover,
$$
f \sim_\p g \ \ \Longrightarrow \ \ \gr(f,\p) = \gr(g,\p).
$$

\begin{thm}\label{graphhom}
Let $G$ be a digraph without right ends whose vertex set $\p$ is a
partition of $\cantor$ and define
$$
X = \bigcup \{a \in \p : a \text{ is a left end of } G\}.
$$
Then, there is a homeomorphism $f$ from $\cantor$ onto $\cantor \backslash X$
such that
$$
\gr(f,\p) = G.
$$
In particular, if $G$ is a digraph without ends, then $f \in \homcantor$.
\end{thm}

\begin{proof} For each $a \in \p$, let $O_a$ be a partition of $a$
whose cardinality is equal to the number of edges of $G$ going out $a$.
For each $b \in \p$ which is not a left end of $G$, let $I_b$ be a
partition of $b$ whose cardinality is equal to the number of edges of $G$
coming into $b$. For each edge $e = \edge{ab} \in G$, we define an
ordered pair $(x_e,y_e)$ such that 
\begin{itemize}
\item $x_e  \in O_{a}$, $y_e \in I_{b}$ and 
\item if $e$ and $e'$ are two distinct edges in $G$, then $x_e \neq x_{e'}$
      and $y_e \neq y _{e'}$. 
\end{itemize}
We note that the second condition above implies that $x_e \cap x_{e'} =
\emptyset$ and $y_e \cap y _{e'} = \emptyset$. Now, let $f$ be a
homeomorphism from $\cantor$ onto $\cantor \backslash X$ such that
$$
f(x_e) = y_e \ \ \text{ for each edge } e \in G.
$$
Then, $\gr(f,\p) = G$.
\end{proof}

\begin{thm}\label{approx}
Let $f \in \contcantor$ and $\q$ be a partition of $\cantor$.
Assume that $G$ is a digraph without right ends and that
$\phi : G \to \gr(f,\q)$ is a surjective graph map. Then:
\begin{itemize}
\item [(a)] There exists a refinement map $\nu : \p \to \q$ and a bijection
      $\psi$ from $\p$ onto the vertex set of $G$ such that
      $\nu = \phi \circ \psi$.
\item [(b)] There exists $g \in \contcantor$ such that
      $\psi : \gr(g,\p) \to G$ is a digraph isomorphism.
      Moreover, if $G$ is a digraph without ends, then we may take such a $g$
      in $\homcantor$.
\item [(c)] For any $g$ as in (b),
      $$
      \tilde{d}(f,g) \leq \mesh(\q) + \mesh(f(\q)),
      $$
      where $f(\q) = \{f(a) : a \in \q\}$.
\end{itemize}
\end{thm}

\begin{proof}
(a) For each $a \in \q$, choose a partition $\p_a$ of $a$ with the same
cardinality as $\phi^{-1}(a)$. Consider the partition
$$
\p = \bigcup_{a \in \q} \p_a
$$
of $\cantor$ and define $\psi$ as a bijection from $\p_a$ onto $\phi^{-1}(a)$
for each $a \in \q$. Then $\psi$ is a bijection from $\p$ onto the vertex
set of $G$ and $\phi \circ \psi : \p \to \q$ is a refinement map.

\smallskip
\noindent (b) There exists a unique digraph $H$ whose vertex set is $\p$
for which $\psi : H \to G$ is a digraph isomorphism.
By Theorem~\ref{graphhom}, there exists $g \in \contcantor$ such that
$$
\gr(g,\p) = H.
$$
Moreover, Theorem~\ref{graphhom} gives $g \in \homcantor$ if $H$
(or equivalently $G$) is a digraph without ends.

\smallskip
\noindent (c) Let us fix $\sigma \in \cantor$. Let $a, b \in \p$ be such that
$\sigma \in a$ and $g(\sigma) \in b$. Then $\edge{ab} \in \gr(g,\p)$,
which implies that
$$\edge{\nu(a)\nu(b)} = \edge{\phi(\psi(a))\phi(\psi(b))} \in \gr(f,\q).$$
Hence, there exists $\tau \in \nu(a)$ such that $f(\tau) \in \nu(b)$.
As $f(\tau), g(\sigma) \in \nu(b)$, we have that
$d(f(\tau),g(\sigma)) \leq \mesh(\q)$.
As $\sigma, \tau \in \nu(a)$, $d(f(\sigma),f(\tau)) \leq \mesh(f(\q))$.
Consequently, $d(f(\sigma),g(\sigma)) \leq \mesh(\q) + \mesh(f(\q))$.
\end{proof}

A digraph $\ell$ is a {\em loop} if the vertex set of $\ell$ is
$\{v_1, \ldots, v_{n}\}$ and the edges of $\ell$ are
$\edge{v_{n}v_1}$ and $\edge{v_iv_{i+1}}$ for $1 \leq i < n$.
In this case, we say that $\ell$ is a loop of {\em length} $n$.

A digraph $B$ is a {\em balloon} if the vertex set of $B$ is the union of
two disjoint sets $p = \{v_1,\ldots,v_s\}$ and $\ell = \{w_1,\ldots,w_t\}$,
and the edges of $B$ are
\begin{itemize}
\item the edges of the path $p$, i.e., $\edge{v_iv_{i+1}}$ for $1 \leq i < s$,
\item the edges of the loop formed by $\ell$, and
\item $\edge{v_sw_1}$.
\end{itemize}
In such case we say that $B$ is a balloon of {\em type $(s,t)$} and we call
$v_1$ the {\em initial vertex of $B$}. Whenever we write a balloon $B$
simply as
$$
B = \{v_1,\ldots,v_s\} \cup \{w_1,\ldots,w_t\},
$$
we implicitly assume that it is the balloon described above.

A digraph $D$ is a {\em dumbbell} if the vertex set of $D$ is the union of
three disjoint sets $l_1=\{u_1, \ldots, u_r\}$, $p=\{v_1, \ldots, v_s\}$ and
$l_2=\{w_1, \ldots, w_t\}$, and the edges of $D$ are
\begin{itemize}
\item the edges of the loops formed by $l_1$ and $l_2$,
\item the edges of the path $p$, i.e., $\edge {v_i v_{i+1}}$
      for $1 \le i < s$, and
\item $\edge {u_1 v_1}$, $\edge {v_sw_1}$. 
\end{itemize} 
In this case we say that $D$ is a dumbbell of {\em type $(r,s,t)$}. If $r=t$,
then we say that the dumbbell is {\em balanced with plate weight $r$}.
We say that $s$ is {\em the length of the bar} of the dumbbell.
Whenever we write a dumbbell $D$ simply as
$$
D = \{u_1,\ldots,u_r\} \cup \{v_1,\ldots,v_s\} \cup \{w_1,\ldots,w_t\},
$$
we implicitly assume that it is the dumbbell described above.

\begin{lem}\label{graphlemma}
{\rm (a)} Let $G$ be a digraph without right ends.
For any edge $e = \edge{uv}$ of $G$, there exist positive integers
$S$ and $M$ so that if $s \geq S$ and $m$ is a positive integer multiple
of $M$, then a balloon of type $(s,m)$ admits a graph map into $G$
such that $u$ is the image of the initial vertex of the balloon.\\
{\rm (b)} Let $G$ be a digraph without ends.
For any edge $e = \edge{uv}$ of $G$, there exist positive integers
$N$, $S$ and $M$ so that if $s \geq S$ and $n, m$ are positive integer
multiples of $N, M$, respectively, then a dumbbell of type $(n,s,m)$
admits a graph map into $G$ such that $e$ is the image of a bar edge
in the dumbbell.
\end{lem}

\begin{proof}
(a) Since $G$ is a digraph without right ends, we can start with edge $e$
and continue a path in $G$ to the right an arbitrary number of steps.
Since there are only finitely many vertices, we obtain a {\em pseudo-balloon}
in $G$, that is, a path
$$
u = u_1, v = u_2,\ldots,u_S,u_{S+1},\ldots,u_{S+M},u_{S+M+1}
$$
with
$$
u_{S+M+1} = u_{S+1}.
$$
We call it a pseudo-balloon because the vertices need not be distinct.
If $s \geq S$ and $m$ is a positive integer multiple of $M$, then we can
extend the bar length from $S$ to $s$ by moving into the pseudo-loop
and continuing aroud it. Then we can go around the pseudo-loop as often as
we want to obtain a pseudo-loop of length $m$. In this way we obtain
a pseudo-balloon in $G$ of the form
$$
u = v_1, v = v_2,\ldots,v_s,v_{s+1},\ldots,v_{s+m},v_{s+m+1} = v_{s+1}.
$$
Now, let $w_i = (v_i,i)$ for $i = 1,\ldots,s+m$ and consider the balloon
$$
B = \{w_1,\ldots,w_s\} \cup \{w_{s+1},\ldots,w_{s+m}\}.
$$
The projection $w_i \in B \mapsto v_i \in G$ is the required graph map.

\smallskip
\noindent (b) Since $G$ is a digraph without ends, we can apply the same
procedure as in (a) also to the left. In this way we obtain a
{\em pseudo-dumbbell} in $G$, that is, a path
$$
u_0,u_1,\ldots,u_N,u_{N+1},\ldots,u_{N+S},
    u_{N+S+1},\ldots,u_{N+S+M},u_{N+S+M+1}
$$
with
$$
u_N = u_0 \ \ \ \text{ and } \ \ \ u_{N+S+M+1} = u_{N+S+1},
$$
so that the original edge $e = \edge{uv}$ is somewhere along the bar.
Now we continue by arguing as in case (a).
\end{proof}

Let us now establish our main graph theoretic result.

\begin{thm}\label{graphthm}
{\rm (a)} Let $G$ be a digraph without right ends. There exist positive
integers $K$, $S$ and $M$ so that if $k \geq K$, $s \geq S$ and $m$ is a
positive integer multiple of $M$, then a digraph consisting of $k$
disjoint balloons of type $(s,m)$ admits a graph map onto $G$.\\
{\rm (b)} Let $G$ be a digraph without ends. There exist positive integers
$K$, $N$, $S$ and $M$ so that if $k \geq K$, $s \geq S$ and $n, m$ are
positive integer multiples of $N, M$, respectively, then a digraph consisting
of $k$ disjoint dumbbells of type $(n,s,m)$ admits a graph map onto $G$.
\end{thm}

\begin{proof}
We shall prove only case (b), since case (a) is analogous.
Let $e_1,\ldots,e_K$ be the edges of $G$. For each $i = 1,\ldots,K$,
the previous lemma associates to the edge $e_i$ positive integers
$N_i$, $S_i$ and $M_i$. Let $N$ and $M$ be the least common multiple
of all the $N_i$'s and of all the $M_i$'s, respectively, and let $S$
be the max of the $S_i$'s. If $n, m$ are multiples of $N, M$,
respectively, and $s \geq S$, then the previous lemma gives us a graph map
from a dumbbell $D_i$ of type $(n,s,m)$ into $G$ hitting $e_i$
($i = 1,\ldots,K$). Clearly, we may assume that the $D_i$'s are
pairwise disjoint. So, the union $H$ of the $D_i$'s admits a graph map
onto $G$. If we want more than $K$ dumbbells, then it is enough to get
as many disjoint copies of $D_1$ (for instance) as we want.
\end{proof}

We are now ready to establish our approximation theorem.

\begin{thm}\label{approxthm}
{\rm (a)} Let $f \in \contcantor$ and $\epsilon > 0$.
There exist positive integers $K$, $S$ and $M$ so that
if $k \geq K$, $s \geq S$ and $m$ is a positive integer multiple of $M$,
then there are $g \in \contcantor$ and a partition $\p$ of $\cantor$ with
$$
\tilde{d}(f,g) < \epsilon \ \ \ \text{ and } \ \ \ \mesh(\p) < \epsilon,
$$
such that the digraph $\gr(g,\p)$ consists of exactly $k$ disjoint
balloons of type $(s,m)$.\\
{\rm (b)} Let $f \in \homcantor$ and $\epsilon > 0$.
There exist positive integers $K$, $N$, $S$ and $M$ so that if $k \geq K$,
$s \geq S$ and $n, m$ are positive integer multiples of $N, M$, respectively,
then there are  $g \in \homcantor$ and a partition $\p$ of $\cantor$ with
$$
\tilde{d}(f,g) < \epsilon \ \ \ \text{ and } \ \ \ \mesh(\p) < \epsilon,
$$
such that the digraph $\gr(g,\p)$ consists of exactly $k$ disjoint
dumbbells of type $(n,s,m)$.
\end{thm}

\begin{proof} Fix $f \in \contcantor$ (resp.\ $f \in \homcantor$)
and $\epsilon > 0$. Using the uniform continuity of $f$, we choose a
partition $\q$ of $\cantor$ such that
$$
\mesh(\q) < \frac{\epsilon}{2} \ \ \ \text{ and } \ \ \
\mesh(f(\q)) < \frac{\epsilon}{2}\cdot
$$
As $\gr(f,\q)$ is a digraph without right ends (resp.\ a digraph without
ends), we can associate to $\gr(f,\q)$ positive integers $K$, $S$ and $M$
(resp.\ $K$, $N$, $S$ and $M$) so that the property described in part (a)
(resp.\ in part (b)) of Theorem~\ref{graphthm} holds. Let $k \geq K$,
$s \geq S$ and $m$ be a positive integer multiple of $M$ (resp.\
$n, m$ be positive integer multiples of $N, M$, respectively).
Then, a digraph $G$ consisting of $k$ disjoint balloons of type $(s,m)$
(resp.\ of $k$ disjoint dumbells of type $(n,s,m)$) admits a graph map $\phi$
onto $\gr(f,\q)$. By Theorem~\ref{approx}, there exist a refinement $\p$ of
$\q$ and $g \in \contcantor$ (resp.\ $g \in \homcantor$) so that
$\gr(g,\p)$ is isomorphic to $G$ and
$$
\tilde{d}(f,g) \leq \mesh(\q) + \mesh(f(\q)) < \epsilon.
$$
This completes the proof.
\end{proof}


\section{Conjugacy Relation}

In this section we present some characterizations of when two continuous
maps of the Cantor space are topologically conjugate to each other.
In particular, we can apply these characterizations to homeomorphisms
of the Cantor space.

We begin by recalling the definition of topological conjugacy.
Suppose $f$ and $g$ are self-maps of spaces $X$ and $Y$, respectively.
We say that $f$ and $g$ are topologically conjugate if there is a
homeomorphism $h$ from $X$ onto $Y$ such that $f = h^{-1}gh$.
We note that topological conjugacy is an equivalence relation.
If $f,g \in \contcantor$ and $h \in \homcantor$, then we simply say that
$f$ and $g$ are {\em conjugates}.

In order to state our next result, let us introduce some terminology.

If $(\p_n)$ is a sequence of partitions of $\cantor$, then we say that
$(\p_n)$ is {\em null} whenever $\mesh(\p_n) \to 0$, and we say that
$(\p_n)$ is {\em decreasing} whenever $\p_{n+1}$ is a refinement of $\p_n$
for every $n \in \nat$. Note that every null sequence of partitions of
$\cantor$ has a decreasing (and null) subsequence.

Suppose $f, g \in \contcantor$, $(\p_n)$ and $(\q_n)$ are decreasing null
sequences of partitions of $\cantor$ and $(\nu_n)$ is a sequence of
isomorphisms
$$
\nu_n : \gr(f,\p_n) \to \gr(g,\q_n).
$$
For each $a \in \p_n$, we define
$$
\nu_m(a) = \bigcup \{\nu_m(b) : b \in \p_m \text{ and } b \subset a\}
  \ \ \ (m \geq n)
$$
and
$$
\tilde{\nu}_n(a) = \bigcup_{m \geq n} \nu_m(a).
$$
We say that the sequence $(\nu_n)$ {\em commutes with refinements}
if the diagrams
$$
\xymatrix{\p_n \ar[rr]^{\nu_n}                      & & \q_n \\
          \p_{n+1} \ar[rr]^{\nu_{n+1}} \ar[u]^{i_n} & & \q_{n+1} \ar[u]_{j_n}}
$$
are commutative, where $i_n$ and $j_n$ denote the refinement maps.
We have the following characterizations of this notion:

\begin{prop}
With the above notations, the following assertions are equivalent:
\begin{itemize}
\item [(i)]   $(\nu_n)$ commutes with refinements;
\item [(ii)]  $\nu_m(a) = \nu_n(a)$ whenever $m \geq n$ and $a \in \p_n$;
\item [(iii)] $\tilde{\nu}_n(a) = \nu_n(a)$ for every $n \in \nat$
      and $a \in \p_n$.
\end{itemize}
\end{prop}

\begin{proof}
(i) $\Rightarrow$ (ii): Suppose $m \geq n$ and $a \in \p_n$.
Since $(\nu_n)$ commutes with refinements, the diagram
$$
\xymatrix{\p_n \ar[rr]^{\nu_n}        & & \q_n \\
          \p_m \ar[rr]^{\nu_m} \ar[u] & & \q_m \ar[u]}
$$
is commutative, where the up arrows indicate the refinement maps.
Consequently,
$$
\nu_m(a) \subset \nu_n(a).
$$
Conversely, take $\sigma \in \nu_n(a)$ and let $c \in \q_m$ be such that
$\sigma \in c$. Since $\nu_m$ is surjective, there exist $b \in \p_m$
such that $\nu_m(b) = c$. Let $a' \in \p_n$ be such that $b \subset a'$.
By what we have just seen, $\nu_m(a') \subset \nu_n(a')$. Hence,
$$
\sigma \in c = \nu_m(b) \subset \nu_m(a') \subset \nu_n(a').
$$
Since $\sigma \in \nu_n(a)$ and $\nu_n$ is injective, $a' = a$.
Thus, $\sigma \in \nu_m(a)$.

\smallskip
\noindent (ii) $\Rightarrow$ (iii): Obvious.

\smallskip
\noindent (iii) $\Rightarrow$ (i): Let $b \in \p_{n+1}$ and put
$a = i_n(b) \in \p_n$. Then
$$
\nu_{n+1}(b) \subset \nu_{n+1}(a) \subset \tilde{\nu}_n(a) = \nu_n(a),
$$
by hypothesis. Hence, $j_n(\nu_{n+1}(b)) = \nu_n(a) = \nu_n(i_n(b))$,
as was to be shown.
\end{proof}

We say that the sequence $(\nu_n)$ {\em asymptotically commutes with
refinements} if
$$
\lim_{n \to \infty} \mesh(\{\tilde{\nu}_n(a) : a \in \p_n\}) = 0.
$$
It follows from the previous proposition that if $(\nu_n)$ commutes with
refinements, then both $(\nu_n)$ and $(\nu_n^{-1})$ asymptotically
commute with refinements.

\begin{thm}\label{conjugacyrelation}
Let $f, g \in \contcantor$. Then the following assertions are equivalent:
\begin{itemize}

\item [(i)] $f$ and $g$ are conjugates;

\item [(ii)] There are decreasing null sequences $(\p_n)$ and $(\q_n)$
      of partitions of $\cantor$ and isomorphisms
      $\nu_n : \gr(f,\p_n) \to \gr(g,\q_n)$ so that the sequence $(\nu_n)$
      commutes with refinements;

\item [(iii)] There are decreasing null sequences $(\p_n)$ and $(\q_n)$
      of partitions of $\cantor$ and isomorphisms
      $\nu_n : \gr(f,\p_n) \to \gr(g,\q_n)$ so that both $(\nu_n)$ and
      $(\nu_n^{-1})$ asymptotically commute with refinements.

\end{itemize}
\end{thm}

\begin{proof}
(i) $\Rightarrow$ (ii): Suppose $f = h^{-1}gh$ for a certain
$h \in \homcantor$. For each $n \in \nat$, let $\p_n = \bn$ and
$\q_n = \{h(a) : a \in \bn\}$. We now observe that for each $n \in \nat$,
$\gr(f,\p_n)$ is isomorphic to $\gr(g,\q_n)$ by the map
$\nu_n : a \in \p_n \mapsto h(a) \in \q_n$. Indeed,
\begin{align*}
\edge{ab} \in \gr(f,\p_n) & \iff f(a) \cap b \neq \emptyset\\
  &\iff hf(a) \cap h(b) \neq \emptyset\\
  & \iff gh(a) \cap h(b) \neq \emptyset\\
  & \iff \edge{h(a)h(b)} \in \gr (g,\q_n).
\end{align*}
Moreover, it is clear that the sequence $(\nu_n)$ commutes with refinements.

\smallskip
\noindent (ii) $\Rightarrow$ (iii): Obvious.

\smallskip
\noindent (iii) $\Rightarrow$ (i): For each $n \in \nat$, we choose an
$h_n \in \homcantor$ such that
$$
h_n(a) = \nu_n(a) \ \ \text{ for every } a \in \p_n.
$$
Then
$$
h_n^{-1}(c) = \nu_n^{-1}(c) \ \ \text{ for every } c \in \q_n.
$$
We shall prove that $(h_n)$ is a Cauchy sequence in $\contcantor$.
For this purpose, let us fix $\epsilon > 0$. Since $(\nu_n)$ asymptotically
commutes with refinements, there exists $n_0 \in \nat$ such that
$$
\mesh(\{\tilde{\nu}_n(a) : a \in \p_n\}) < \epsilon \ \ \text{ whenever }
  n \geq n_0.
$$
Take $\sigma \in \cantor$ and $m \geq n \geq n_0$. Let $a \in \p_n$ and
$b \in \p_m$ be such that $\sigma \in a$ and $\sigma \in b$. Then,
$$
h_n(\sigma) \in \nu_n(a) \subset \tilde{\nu}_n(a) \ \ \ \text{ and } \ \ \
h_m(\sigma) \in \nu_m(b) \subset \nu_m(a) \subset \tilde{\nu}_n(a),
$$
and so
$$
d(h_m(\sigma),h_n(\sigma)) < \epsilon.
$$
By completeness, the sequence $(h_n)$ must converge to a function
$h$ in $\contcantor$. Since $(\nu_n^{-1})$ also asymptotically commutes
with refinements, we may apply the same argument to the sequence
$(h_n^{-1})$ and conclude that this sequence converges to a function
$t$ in $\contcantor$. Since $h_n \circ h_n^{-1} = h_n^{-1} \circ h_n = I$
(the identity map of $\cantor$) for every $n \in \nat$,
$h \circ t = t \circ h = I$. Thus, $h \in \homcantor$.

Now, given $\sigma \in \cantor$ and $n \in \nat$, let $a, b \in \p_n$ be
such that $\sigma \in a$ and $f(\sigma) \in b$.
Then, $\edge{ab} \in \gr(f,\p_n)$, which implies that
$\edge{\nu_n(a)\nu_n(b)} \in \gr(g,\q_n)$, that is,
$$
g(\nu_n(a)) \cap \nu_n(b) \neq \emptyset.
$$
Since $h_nf(\sigma) \in \nu_n(b)$ and $gh_n(\sigma) \in g(\nu_n(a))$,
we obtain
$$
d(h_nf(\sigma),gh_n(\sigma)) \leq \mesh(\q_n) + \mesh(g(\q_n)).
$$
Hence, by letting $n \to \infty$, we conclude that $hf = gh$, that is,
$f = h^{-1}gh$.
\end{proof}

In order to state our next characterizations of the conjugacy relation,
we need to introduce some further terminology.

Suppose $f, g \in \contcantor$, $(\p_n)$ and $(\q_n)$ are decreasing null
sequences of partitions of $\cantor$ and $(\nu_n)$ is a sequence of
surjective graph maps with
$$
\nu_n : \gr(f,\p_n) \to \gr(g,\q_n) \ \ \text{ for odd } n
$$
and
$$
\nu_n : \gr(g,\q_n) \to \gr(f,\p_n) \ \ \text{ for even } n.
$$
For each odd $n$ and each $a \in \p_n$, we define
$$
\nu_m(a) = \bigcup \{\nu_m(b) : b \in \p_m \text{ and } b \subset a\}
  \ \ \ (m \geq n, m \text{ odd}),
$$
$$
\nu_m^{-1}(a) = \bigcup \{c : c \in \q_m \text{ and } \nu_m(c) \subset a\}
  \ \ \ (m > n, m \text{ even})
$$
and
$$
\ov{\nu}_n(a) = \bigcup \{\nu_m(a) : m \geq n, m \text{ odd}\} \cup
                \bigcup \{\nu_m^{-1}(a) : m > n, m \text{ even}\}.
$$
Analogously, we define $\nu_m(c)$ ($m \geq n$, $m$ even), $\nu_m^{-1}(c)$
($m > n$, $m$ odd) and $\ov{\nu}_n(c)$ for even $n$ and $c \in \q_n$.
We say that the sequence $(\nu_n)$ {\em commutes with refinements}
if the diagrams
$$
\xymatrix{\p_n \ar[rr]^{\nu_n} & & \q_n \\
       \p_{n+1} \ar[u]^{i_n} & & \q_{n+1} \ar[ll]_{\nu_{n+1}} \ar[u]_{j_n}}
\hspace{.5in}
\xymatrix{\p_n & & \q_n \ar[ll]_{\nu_n} \\
       \p_{n+1} \ar[u]^{i_n} \ar[rr]^{\nu_{n+1}} & & \q_{n+1} \ar[u]_{j_n}}
$$

\vspace{-.2in}

$$
n \text{ odd} \hspace{1.6in} n \text{ even}
$$
are commutative, where $i_n$ and $j_n$ denote the refinement maps.
We have the following characterizations of this concept:

\begin{prop}
With the above notations, the following assertions are equivalent:
\begin{itemize}

\item [(i)] $(\nu_n)$ commutes with refinements;

\item [(ii)] The following inclusions hold:
     \begin{itemize}
     \item [$\bullet$] $\nu_m(a) \subset \nu_n(a)$ whenever
           $n$ is odd, $m$ is odd, $m \geq n$ and $a \in \p_n$.
     \item [$\bullet$] $\nu_m^{-1}(a) \subset \nu_n(a)$ whenever
           $n$ is odd, $m$ is even, $m > n$ and $a \in \p_n$.
     \item [$\bullet$] $\nu_m(c) \subset \nu_n(c)$ whenever
           $n$ is even, $m$ is even, $m \geq n$ and $c \in \q_n$.
     \item [$\bullet$] $\nu_m^{-1}(c) \subset \nu_n(c)$ whenever
           $n$ is even, $m$ is odd, $m > n$ and $c \in \q_n$.
     \end{itemize}

\item [(iii)] The following equalities hold:
      \begin{itemize}
      \item [$\bullet$] $\ov{\nu}_n(a) = \nu_n(a)$
            whenever $n$ is odd and $a \in \p_n$.
      \item [$\bullet$] $\ov{\nu}_n(c) = \nu_n(c)$
            whenever $n$ is even and $c \in \q_n$.
      \end{itemize}

\end{itemize}
\end{prop}

\begin{proof}
(i) $\Rightarrow$ (ii): Let $n, m, k \in \nat$ be such that $n$ and $m$
are odd, $k$ is even, $m \geq n$ and $k > n$. Since $(\nu_n)$ commutes
with refinements, the diagrams
$$
\xymatrix{\p_n \ar[rr]^{\nu_n}        & & \q_n \\
          \p_m \ar[rr]^{\nu_m} \ar[u] & & \q_m \ar[u]}
\hspace{.5in}
\xymatrix{\p_n \ar[rr]^{\nu_n} & & \q_n \\
          \p_k \ar[u]          & & \q_k \ar[ll]_{\nu_k} \ar[u]}
$$
are commutative, where the up arrows indicate the refinement maps.
This implies that $\nu_m(a) \subset \nu_n(a)$ and
$\nu_k^{-1}(a) \subset \nu_n(a)$ for every $a \in \p_n$,
which proves the first two inclusions in (ii).
The other two inclusions are proved in a similar way.

\smallskip
\noindent (ii) $\Rightarrow$ (iii): Obvious.

\smallskip
\noindent (iii) $\Rightarrow$ (i): Assume $n$ odd and let us prove
that
$$
j_n = \nu_n \circ i_n \circ \nu_{n+1}.
$$
Take $c \in \q_{n+1}$ and put $b = \nu_{n+1}(c) \in \p_{n+1}$ and
$a = i_n(b) \in \p_n$.
Since $n+1$ is even and $\nu_{n+1}(c) = b \subset a$, we have that
$$
c \subset (\nu_{n+1})^{-1}(a) \subset \ov{\nu}_n(a) = \nu_n(a).
$$
Thus, $j_n(c) = \nu_n(a)$, that is, $j_n(c) = \nu_n(i_n(\nu_{n+1}(c)))$.
For $n$ even the proof is analogous.
\end{proof}

We say that the sequence $(\nu_n)$ {\em asymptotically commutes with
refinements} if
$$
\lim_{n \to \infty} \mesh(\{\ov{\nu}_{2n-1}(a) : a \in \p_{2n-1}\}) = 0
$$
and
$$
\lim_{n \to \infty} \mesh(\{\ov{\nu}_{2n}(c) : c \in \q_{2n}\}) = 0.
$$
By the previous proposition, if $(\nu_n)$ commutes with refinements,
then $(\nu_n)$ asymptotically commutes with refinements.

\begin{thm}\label{ConjugacyRelation}
Let $f, g \in \contcantor$. Then the following assertions are equivalent:
\begin{itemize}

\item [(i)] $f$ and $g$ are conjugates;

\item [(ii)] There are decreasing null sequences $(\p_n)$ and $(\q_n)$
      of partitions of $\cantor$ and surjective graph maps
      $\nu_{2n-1} : \gr(f,\p_{2n-1}) \to \gr(g,\q_{2n-1})$ and
      $\nu_{2n} : \gr(g,\q_{2n}) \to \gr(f,\p_{2n})$
      so that the sequence $(\nu_n)$ commutes with refinements;

\item [(iii)] There are decreasing null sequences $(\p_n)$ and $(\q_n)$
      of partitions of $\cantor$ and surjective graph maps
      $\nu_{2n-1} : \gr(f,\p_{2n-1}) \to \gr(g,\q_{2n-1})$ and
      $\nu_{2n} : \gr(g,\q_{2n}) \to \gr(f,\p_{2n})$
      so that the sequence $(\nu_n)$ asymptotically commutes with refinements.

\end{itemize}
\end{thm}

\begin{proof}
(i) $\Rightarrow$ (ii): By the implication (i) $\Rightarrow$ (ii) in
Theorem~\ref{conjugacyrelation}, there are decreasing null sequences
$(\p_n)$ and $(\q_n)$ of partitions of $\cantor$ and isomorphisms
$\theta_n : \gr(f,\p_n) \to \gr(g,\q_n)$ so that the sequence $(\theta_n)$
commutes with refinements. So, it is enough to define $\nu_n = \theta_n$
for odd $n$ and $\nu_n = \theta_n^{-1}$ for even $n$.

\smallskip
\noindent (ii) $\Rightarrow$ (iii): Obvious.

\smallskip
\noindent (iii) $\Rightarrow$ (i): For each odd $n$ and each even $m$,
we define
$$
\alpha_n = \mesh(\{\ov{\nu}_n(a) : a \in \p_n\}) \ \ \ \text{ and } \ \ \
 \beta_m = \mesh(\{\ov{\nu}_m(c) : c \in \q_m\}).
$$
Since $(\nu_n)$ asymptotically commutes with refinements,
$$
(\alpha_n)_{n \text{ odd}} \to 0 \ \ \ \text{ and } \ \ \
(\beta_m)_{m \text{ even}} \to 0.
$$

For each odd $n$, we choose an $h_n \in \homcantor$ such that
$$
\bigcup_{a \in \nu_n^{-1}(c)} h_n(a) = c \ \ \text{ for every } c \in \q_n.
$$
Note that
$$
h_n(a) \subset \nu_n(a) \ \ \text{ for every } a \in \p_n.
$$
Given $\sigma \in \cantor$ and $m \geq n$ with $m$ and $n$ odd,
let $a \in \p_n$ and $b \in \p_m$ be such that $\sigma \in a$ and
$\sigma \in b$. Then
$$
h_n(\sigma) \in \nu_n(a) \subset \ov{\nu}_n(a) \ \ \ \text{ and } \ \ \
h_m(\sigma) \in \nu_m(b) \subset \nu_m(a) \subset \ov{\nu}_n(a),
$$
and so $d(h_m(\sigma),h_n(\sigma)) \leq \alpha_n$. This proves that
$(h_n)_{n \text{ odd}}$ is a Cauchy sequence in $\contcantor$ and so
it converges to a certain $h \in \contcantor$. Similarly, for each even $n$,
we choose a $t_n \in \homcantor$ such that
$$
\bigcup_{c \in \nu_n^{-1}(a)} t_n(c) = a \ \ \text{ for every } a \in \p_n.
$$
Then
$$
t_n(c) \subset \nu_n(c) \ \ \text{ for every } c \in \q_n.
$$
By arguing as before, we see that $(t_n)_{n \text{ even}}$ converges
to a certain $t \in \contcantor$.

Now, let $\sigma \in \cantor$ and $n$ be odd. Let $c \in \q_{n+1}$ and
$a \in \p_n$ be such that $\sigma \in c$ and $t_{n+1}(\sigma) \in a$.
As $t_{n+1}(\sigma) \in t_{n+1}(c) \subset \nu_{n+1}(c) \in \p_{n+1}$
and $t_{n+1}(\sigma) \in a \in \p_n$, we must have $\nu_{n+1}(c) \subset a$.
Hence,
$$
\sigma \in c \subset (\nu_{n+1})^{-1}(a) \subset \ov{\nu}_n(a).
$$
On the other hand,
$$
h_n(t_{n+1}(\sigma)) \in h_n(a) \subset \nu_n(a) \subset \ov{\nu}_n(a).
$$
Therefore, $d(h_n(t_{n+1}(\sigma)),\sigma) \leq \alpha_n$.
This proves that $(h_n \circ t_{n+1})_{n \text{ odd}} \to I$,
the identity map of $\cantor$. Analogously,
$(t_n \circ h_{n+1})_{n \text{ even}} \to I$.
Thus, $h \circ t = t \circ h = I$, and so $h \in \homcantor$.

Finally, by arguing exactly as in the last paragraph of the proof of
Theorem~\ref{conjugacyrelation} (but considering only odd $n$),
we conclude that $f = h^{-1} \circ g \circ h$.
\end{proof}


\section{Generics and Applications to Dynamics}

Suppose that $h \in \homcantor$, $\p$ is a partition of $\cantor$ and
$D$ is a component of $\gr(h,\p)$ which is a dumbbell. Write
$$
D = \{u_1,\ldots,u_r\} \cup \{v_1,\ldots,v_s\} \cup \{w_1,\ldots,w_t\},
$$
with usual labeling. We say that {\em $D$ contains a left loop of $h$}
(resp.\ {\em a right loop of $h$}) if there is a nonempty clopen subset
$a$ of $u_1$ (resp.\ of $w_1$) such that $h^r(a) = a$ (resp.\ $h^t(a) = a$).

By using our methods, we shall now give a simple geometric/graph theoretic
description of the comeager conjugacy class of $\homcantor$.

\begin{thm}\label{ConjugacyClass}
Let $S$ be the set of all $h \in \homcantor$ with the following property:

\smallskip
\noindent {\rm (P)} For every $m \in \nat$, there are a partition $\p$ of
$\cantor$ of mesh $< 1/m$ and a multiple $q \in \nat$ of $m$ such that every
component of $\gr(h,\p)$ is a balanced dumbbell with plate weight $q!$
that contains both a left and a right loop of $h$.

\smallskip
\noindent Then, $S$ is a comeager conjugacy class of $\homcantor$.
\end{thm}

\begin{proof}
For each $m \in \nat$, let $S_m$ be the set of all $h \in \homcantor$ that
satisfies the property contained in (P) for this particular $m$.
If we fix a partition $\q$ of $\cantor$, then the map $f \to \gr(f,\q)$
is locally constant on $\contcantor$, because $f \sim_\q g$ implies
$\gr(f,\q) = \gr(g,\q)$. Moreover, if $a$ is clopen in $\cantor$ and
$n \in \nat$, then the condition $f^n(a) \subset a$ is an open condition
on $\contcantor$. By applying this to the inverse, we see that
$f^n(a) = a$ is an open condition on $\homcantor$. Therefore,
each $S_m$ is open in $\homcantor$. Let us prove that each $S_m$ is also
dense in $\homcantor$. For this purpose, fix $m \in \nat$, $f \in \homcantor$
and $\epsilon > 0$. By applying the approximation theorem
(Theorem~\ref{approxthm}(b)) with $\min\{\frac{\epsilon}{2},\frac{1}{m}\}$
in place of $\epsilon$, we obtain positive integers $K$, $N$, $S$ and $M$
with the properties described in the theorem. Choose a multiple $q \geq 2$
of $m$ such that $q!$ is a multiple of both $N$ and $M$. Then, with
$k = K$, $s = S$ and $n = m = q!$, Theorem~\ref{approxthm}(b) gives us
a $g \in \homcantor$ and a partition $\p$ of $\cantor$ with
$$
\tilde{d}(f,g) < \frac{\epsilon}{2} \ \ \ \text{ and } \ \ \
\mesh(\p) < \min\big\{\frac{\epsilon}{2},\frac{1}{m}\big\},
$$
such that $\gr(g,\p)$ is a digraph whose components are balanced dumbbells
with plate weight $q!$. Now, we define $h \in \homcantor$ in the following
way: for each component (dumbbell)
$$
D =\{u_1,\ldots,u_{q!}\} \cup \{v_1,\ldots,v_s\} \cup \{w_1,\ldots,w_{q!}\}
$$
of $\gr(g,\p)$ (with usual labeling), we choose nonempty proper clopen
subsets $a$ of $g^{-1}(u_2)$ and $b$ of $g(w_{q!})$, define $h$ on $u_{q!}$
so that
$$
h(g^{q!-1}(a)) = a \ \ \ \text{ and } \ \ \
h(u_{q!} \backslash g^{q!-1}(a)) = u_1 \backslash a,
$$
define $h$ on $w_{q!}$ so that
$$
h(g^{q!-1}(b)) = b \ \ \ \text{ and } \ \ \
h(w_{q!} \backslash g^{q!-1}(b)) = g(w_{q!}) \backslash b,
$$
and put $h = g$ on the remaining vertices of $D$. Then,
$\gr(h,\p) = \gr(g,\p)$ and each component (dumbbell) of $\gr(h,\p)$
contains both a left and a right loop of $h$, which shows that $h \in S_m$.
Moreover, $\tilde{d}(g,h) < \frac{\epsilon}{2}$ because
$\mesh(\p) < \frac{\epsilon}{2}\cdot$
Therefore, $\tilde{d}(f,h) < \epsilon$, proving that $S_m$ is dense in
$\homcantor$. Thus, $S = \bigcap S_m$ is a comeager subset of $\homcantor$.

In order to prove that $S$ is a conjugacy class, we shall begin by making
several important remarks and introducing further terminology.

Suppose that $h \in \homcantor$, $\p$ is a partition of $\cantor$ and
$D$ is a component of $\gr(h,\p)$ which is a dumbbell. Write
$$
D = \{u_1,\ldots,u_r\} \cup \{v_1,\ldots,v_s\} \cup \{w_1,\ldots,w_t\},
$$
with usual labeling. If we replace the set $u_1$ of $\p$ by the sets
$h^{-1}(u_2)$ and $h^{-1}(v_1)$, we obtain a refinement $\p'$ of $\p$
such that $\gr(h,\p')$ has the following dumbbell as a component:
$$
D' = \{u_r,h^{-1}(u_2),u_2,\ldots,u_{r-1}\} \cup
     \{h^{-1}(v_1),v_1,\ldots,v_s\} \cup \{w_1,\ldots,w_t\}.
$$
We call this procedure the {\em method of increasing the bar of the dumbbell
to the left}. Note that this method doesn't change the plate weights but
increase the bar length by $1$. Similarly, by breaking $w_1$ in the parts
$h(v_s)$ and $h(w_t)$, we obtain a refinement $\p''$ of $\p$ such that
$\gr(h,\p'')$ has the following dumbbell as a component:
$$
D'' = \{u_1,\ldots,u_r\} \cup \{v_1,\ldots,v_s,h(v_s)\} \cup
      \{w_2,\ldots,w_t,h(w_t)\}.
$$
This is called the {\em method of increasing the bar of the dumbbell to the
right}. Hence, by applying these methods repeatedly, we can make the bar of
the dumbbell $D$ increase to the left and/or to the right as much as we want.
This remark will be quite important in the sequel.

Let $h \in \homcantor$. We say that a partition $\p$ of $\cantor$ is
{\em $h$-regular} if each component of $\gr(h,\p)$ is a balanced dumbbell
that contains both a left and a right loop of $h$ and all components of
$\gr(h,\p)$ have the same plate weight (denoted $w(h,\p)$). Note that the
methods of increasing the bars of the dumbbells transform $h$-regular
partitions in $h$-regular partitions. If $h \in S$ then there are
$h$-regular partitions $\p$ such that $\mesh(\p)$ is as small as we want
and $w(h,\p)$ is a multiple of any positive integer we want.

Suppose $\p$ and $\p'$ are $h$-regular partitions. If $\mesh(\p')$ is
sufficiently small, then $\p'$ is necessarily a refinement of $\p$.
Assume that this is the case. Then, each component $D'$ of $\gr(h,\p')$
must be {\em contained} in some component $D$ of $\gr(h,\p)$, in the sense
that the union of all vertices of $D'$ is contained in the union of all
vertices of $D$. Moreover, $w(h,\p')$ is necessarily a multiple of $w(h,\p)$.
Write
$$
D = \{u_1,\ldots,u_q\} \cup \{v_1,\ldots,v_\ell\} \cup \{w_1,\ldots,w_q\}
$$
and
$$
D' = \{u'_1,\ldots,u'_{q'}\} \cup \{v'_1,\ldots,v'_{\ell'}\} \cup
     \{w'_1,\ldots,w'_{q'}\}
$$
($q = w(h,\p)$ and $q' = w(h,\p')$), with usual labeling. There are
three possibilities:

\smallskip
\noindent 1) $\bigcup\{a: a \in D'\} \subset u_1 \cup \ldots \cup u_q$:

\smallskip
By applying the methods of increasing the bar of $D'$, we may assume
$u'_1, w'_1 \subset u_1$. Under this assumption, we say that $D'$ is a
{\em subdumbbell of $D$ of type $1$}.

\smallskip
\noindent 2) $\bigcup\{a: a \in D'\} \subset w_1 \cup \ldots \cup w_q$:

\smallskip
Similarly, we may assume $u'_1, w'_1 \subset w_1$ in this case. Under this
assumption, we say that $D'$ is a {\em subdumbbell of $D$ of type $2$}.

\smallskip
\noindent 3) $\bigcup\{a: a \in D'\}$ meets $v_1 \cup \ldots \cup v_\ell$:

\smallskip
In this case, there must exist an integer $r \geq 0$ such that
$$
v'_{r+1} \subset v_1, \ v'_{r+2} \subset v_2, \ldots ,
v'_{r+\ell} \subset v_\ell,
$$
$$
u'_1 \cup \ldots \cup u'_{q'} \cup v'_1 \cup \ldots \cup v'_r
  \subset u_1 \cup \ldots \cup u_q
$$
and
$$
v'_{r+\ell+1} \cup \ldots \cup v'_{\ell'} \cup w'_1 \cup \ldots \cup w'_{q'}
  \subset w_1 \cup \ldots \cup w_q.
$$
By applying the methods of increasing the bar of $D'$, we may assume
$u'_1 \subset u_1$ and $w'_1 \subset w_1$. With this assumption, both
the number $r$ of $v'_j$ to the left of $v'_{r+1}$ and
the number $\ell' - \ell - r$ of $v'_j$ to the right of $v'_{r+\ell}$
are multiples of $q$. Therefore, by applying the methods of increasing
the bar of $D'$ again, we may assume $r = \ell' - \ell - r$.
Geometrically, this equality give us symmetry: $v'_{r+1},\ldots,v'_{r+\ell}$
lie in the center of the bar of $D'$. Under these assumptions, we say that
$D'$ is a {\em subdumbbell of $D$ of type~$3$}.

\smallskip
The previous discussion suggests the following definition:
if $\p$ and $\p'$ are $h$-regular partitions, we say that $\p'$ is
an {\em $h$-subpartition of $\p$} if $\p'$ is a refinement of $\p$ and
every component of $\gr(h,\p')$ is a subdumbbell (of type 1, 2 or 3)
of a component of $\gr(h,\p)$. We have seen that if $h \in S$, then every
$h$-regular partition $\p$ has $h$-subpartitions $\p'$ such that
$\mesh(\p')$ is as small as we want and $w(h,\p')$ is a multiple of
any positive integer we want.

Suppose $\p'$ is an $h$-subpartition of $\p$ and let $D$ be a component
of $\gr(h,\p)$. Then $D$ can be thought of as the union of its subdumbbells
relative to $\p'$. Clearly, there must exist at least one subdumbbell of $D$
of type 3. Since $D$ has both a left and a right loop of $h$, there must
also exist subdumbbells of $D$ of types $1$ and $2$ provided $\mesh(\p')$
is sufficiently small. More precisely, we can make the number of
subdumbbells of $D$ of type $1$ (resp.\ type $2$, type $3$) as large as
we want by choosing $\p'$ with $\mesh(\p')$ small enough.

We are now in position to prove that $S$ is a conjugacy class.
If $f \in S$ and $g \in \homcantor$ is conjugate to $f$, then it is easy
to verify that $g \in S$. Let us fix $f, g \in S$. It remains to prove that
$f$ and $g$ are conjugates. In view of Theorem~\ref{ConjugacyRelation},
it is enough to construct sequences $(\p_n)$, $(\q_n)$ and $(\nu_n)$
with the properties described in part (ii) of the theorem.

In order to construct $\p_1$, $\q_1$ and $\nu_1$, we begin by taking a
$g$-regular partition $\q_1$ with $\mesh(\q_1) < 1$. Then, we take an
$f$-regular partition $\p_1$ such that $\mesh(\p_1) < 1$, $w(f,\p_1)$
is a multiple of $w(g,\q_1)$ and the set $A$ of all components of
$\gr(f,\p_1)$ has cardinality greather than or equal to that of the
set $B$ of all components of $\gr(g,\q_1)$. By applying the methods
of increasing the bars of the dumbbells, we may assume that all
dumbbells in $A \cup B$ have the same bar length. Choose a surjection
$\phi : A \to B$. For each $D \in A$, we define $\nu_1$ on $D$ as the
unique surjection from $D$ onto $\phi(D)$ that maps the bar of $D$
onto the bar of $\phi(D)$ and satisfies the relation
$$
\edge{ab} \in \gr(f,\p_1) \ \Longrightarrow \
\edge{\nu_1(a)\nu_1(b)} \in \gr(g,\q_1) \ \ \ \ \ (a, b \in D).
$$
In this way, we obtain a surjective graph map
$\nu_1 : \gr(f,\p_1) \to \gr(g,\q_1)$.

In order to construct $\p_2$, $\q_2$ and $\nu_2$, we begin by taking an
$f$-subpartition $\p_2$ of $\p_1$ such that $\mesh(\p_2) < 1/2$ and
every component of $\gr(f,\p_1)$ has subdumbbells of types 1, 2 and 3
relative to $\p_2$. Then, we take a $g$-subpartition $\q_2$ of $\q_1$
such that $\mesh(\q_2) < 1/2$, $w(g,\q_2)$ is a multiple of $w(f,\p_2)$
and every component of $\gr(g,\q_1)$ has subdumbbells of types 1, 2 and 3
relative to $\q_2$. Fix a component $D$ of $\gr(g,\q_1)$ and let
$\{D_1,\ldots,D_r\}$ be the set of all components $D_j$ of $\gr(f,\p_1)$
such that $\nu_1(D_j) = D$. We can divide the set of all subdumbbells $D'$
of $D$ relative to $\q_2$ in three sets $A_1$, $A_2$ and $A_3$, according to
whether $D'$ is of type 1, 2 or 3, respectively. Similarly, for each
$1 \leq j \leq r$, the set of all subdumbbells of $D_j$ relative to $\p_2$
is a union of three disjoint sets $B_{j,1}$, $B_{j,2}$ and $B_{j,3}$.
We may assume that $\q_2$ was chosen so that
$$
{\rm Card}\, A_i \geq {\rm Card}(B_{1,i} \cup \ldots \cup B_{r,i})
\ \ \text{ for } i = 1, 2, 3.
$$
Hence, we may choose a surjection
$\phi_i : A_i \to B_{1,i} \cup \ldots \cup B_{r,i}$ ($i = 1, 2, 3$).
Moreover, by using the facts that the dumbbells $D,D_1,\ldots,D_r$
have the same bar length and that each number in the finite sequence
$w(g,\q_2)$, $w(f,\p_2)$, $w(f,\p_1)$, $w(g,\q_1)$ is a multiple of its
successor, and by applying the methods of increasing the bars of the
dumbbells, we may assume that all dumbbells in
$A_i \cup B_{1,i} \cup \ldots \cup B_{r,i}$ have the same bar
length (but the bar length may depend on $i$). Now, for each
$i \in \{1,2,3\}$ and each $D' \in A_i$,
we define $\nu_2$ on $D'$ as the unique surjection from $D'$ onto
$\phi_i(D')$ that maps the bar of $D'$ onto the bar of $\phi_i(D')$
and satisfies the relation
$$
\edge{ab} \in \gr(g,\q_2) \ \Longrightarrow \
\edge{\nu_2(a)\nu_2(b)} \in \gr(f,\p_2) \ \ \ \ \ (a, b \in D').
$$
With this definition, note that
$$
j_1(a) = \nu_1(i_1(\nu_2(a))) \ \ \text{ for every } a \in D',
$$
where $i_1 : \p_2 \to \p_1$ and $j_1 : \q_2 \to \q_1$ are the refinement maps.
Indeed, it is easy to see that this is true by using the following
elementary arithmetic fact: if $n_1, n_2, m_1, m_2 \in \nat$,
$m_1$ divides $n_1$, $n_1$ divides $n_2$ and $n_2$ divides $m_2$,
then the diagram
$$
\xymatrix{\integer_{n_1} \ar[rr] & & \integer_{m_1} \\
         \integer_{n_2} \ar[u]   & & \integer_{m_2} \ar[ll] \ar[u]}
$$
is commutative, where $\integer_n \to \integer_m$ denotes the mod $m$ map.
By defining $\nu_2$ in this way for each component $D$ of $\gr(g,\q_1)$,
we obtain a surjective graph map $\nu_2 : \gr(g,\q_2) \to \gr(f,\p_2)$
such that $j_1 = \nu_1 \circ i_1 \circ \nu_2$.

Now, we apply exactly the same procedure to construct $\p_3$, $\q_3$
and $\nu_3$ (but with $\p_3$ and $\q_3$ in place of $\q_2$ and $\p_2$,
respectively), and so on. This completes the proof that $f$ and $g$
are conjugates.
\end{proof}

Our goal in the remaining of the present section is to show how the
description of the comeager conjugacy class of $\homcantor$ given
by Theorem~\ref{ConjugacyClass} can be used to establish with little
effort very precise properties of the homeomorphisms of this class.
This will make clear that this description is quite useful and very
easy to use. We observe that any comeager dynamical property of an
element of $\homcantor$ is automatically satisfied by all elements
of the comeager conjugacy class. Let us also mention that several
of these properties can also be derived from the explicit construction
of Akin, Glasner and Weiss \cite{agw}.

Let us begin by considering the notion of chaos. As mentioned in the
introduction, there are several different notions of chaos, Li-Yorke chaos
being the weakest of them.
Let us recall that a pair $(x,y)$ is a Li-Yorke pair for a function $f$ if
$$
\liminf_{n \to \infty} d(f^n(x),f^n(y)) = 0 \ \ \ \text{ and } \ \ \
\limsup_{n \to \infty} d(f^n(x),f^n(y)) > 0.
$$
A function $f$ is Li-Yorke chaotic if there is an uncountable set $S$
such that $(x,y)$ is a Li-Yorke pair for $f$ whenever $x$ and $y$ are
distinct points in $S$. It was proved in \cite{gw} that 
the set of homeomorphisms in $\homcantor$ which have topological entropy
zero is comeager in $\homcantor$. Hence, a homeomorphism chosen at
``random'' is not chaotic in the sense of entropy. We shall now see that
the elements of the comeager conjugacy class are not even Li-Yorke chaotic.
In fact, a much stronger assertion can be made.

\begin{thm}\label{LiYorke}
No element of the comeager conjugacy class of $\homcantor$ has a
Li-Yorke pair.
\end{thm}

\begin{proof}
Let $h$ be an element of the comeager conjugacy class of $\homcantor$.
Then, $h$ satisfies property (P) of Theorem~\ref{ConjugacyClass}.
Assume $(\sigma,\tau)$ is a Li-Yorke pair for $h$ and choose $m \in \nat$
such that
$$
\frac{1}{m} < \limsup_{n \to \infty} d(h^n(\sigma),h^n(\tau)).
$$
Then, there is a partition $\p$ of $\cantor$ of mesh $< 1/m$ such that
every component of $\gr(h,\p)$ is a dumbbell. By our choice of $m$,
there must exist infinitely many $n$'s such that $h^n(\sigma)$ and $h^n(\tau)$
lie in different sets of $\p$. On the other hand, since
$\liminf_{n \to \infty} d(h^n(\sigma),h^n(\tau)) = 0$, there must exist
infinitely many $n$'s such that both $h^n(\sigma)$ and $h^n(\tau)$
lie in the same set of $\p$. But this impossible since every component
of $\gr(h,\p)$ is a dumbbell and two points in the same vertex of such a
dumbbell can be mapped into different vertices only once.
\end{proof}

Let $h : X \to X$ be a homeomorphism, where $X$ is a metric space.
A sequence $(x_n)_{n \in \integer}$ is a $\delta$-pseudotrajectory
($\delta > 0$) of $h$ if
$$
d(h(x_n),x_{n+1}) \leq \delta \ \text{ for every } n \in \integer.
$$
Recall that $h$ is said to have the shadowing property
\cite{bowen1},~\cite{bowen2}
(also called pseudo-orbit tracing property) if for every
$\epsilon > 0$ there exists $\delta > 0$ such that every
$\delta$-pseudotrajectory $(x_n)_{n \in \integer}$ of $h$ is
$\epsilon$-shadowed by a real trajectory of $h$, i.e., there exists $x \in X$
such that
$$
d(x_n,h^n(x)) < \epsilon \ \text{ for every } n \in \integer.
$$
Recall also that $h$ is said to have the weak shadowing property \cite{cp} if
for every $\epsilon > 0$ there exists $\delta > 0$ such that for
every $\delta$-pseudotrajectory $(x_n)_{n \in \integer}$ of $h$
there exists $x \in X$ such that the set $\{x_n : n \in \integer\}$
is contained in the $\epsilon$-neighborhood of the orbit
$\{h^n(x) : n \in \integer\}$. It was proved in \cite{mazur} that there
is a comeager subset of $\homcantor$, each element of which has the weak shadowing property.
We shall now see that it actually has the shadowing property.

\begin{thm}\label{Shadowing}
Each element of the comeager conjugacy class of $\homcantor$ has the shadowing property.
\end{thm}

\begin{proof}
Let $h$ be an element of the comeager conjugacy class of $\homcantor$.
Fix $\epsilon > 0$ and choose
$m \in \nat$ with $1/m < \epsilon$. Then, there is a partition $\p$ of
$\cantor$ of mesh $< 1/m$ such that every component of $\gr(h,\p)$ is a
dumbbell. Let $\delta$ be the minimum distance between two distinct elements
of $\p$. To each sequence $X = (\sigma_n)_{n \in \integer}$ in $\cantor$,
we associate the sequence $S(X) = (S_n(X))_{n \in \integer}$ in $\p$ which
satisfies $\sigma_n \in S_n(X)$ for every $n \in \integer$.
Let $W = (\tau_n)_{n \in \integer}$ be a $\delta$-pseudotrajectory of $h$
and let us prove that it can be $\epsilon$-shadowed by a real trajectory
of $h$. Let
$$
D = \{u_1,\ldots,u_r\} \cup \{v_1,\ldots,v_s\} \cup \{w_1,\ldots,w_t\}
$$
be the dumbbell in $\gr(h,\p)$ (with usual labeling) that contains a
vertex containing $\tau_0$. By our choice of $\delta$, there are only
three possibilities for the sequence $S(W)$, namely:
\begin{itemize}
\item [(1)] $(\ldots,u_1,\ldots,u_r,u_1,\ldots,u_r,v_1,\ldots,v_s,w_1,\ldots,
  w_t,w_1,\ldots,w_t,\ldots)$, or
\item [(2)] $(\ldots,u_1,\ldots,u_r,u_1,\ldots,u_r,u_1,\ldots,u_r,\ldots)$, or
\item [(3)] $(\ldots,w_1,\ldots,w_t,w_1,\ldots,w_t,w_1,\ldots,w_t,\ldots)$.
\end{itemize}
Since $\p$ has mesh $< \epsilon$, it is enough to find real trajectories
$$
X = (h^n(x))_{n \in \integer}, \ \
Y = (h^n(y))_{n \in \integer} \ \ \text{ and } \ \
Z = (h^n(z))_{n \in \integer}
$$
such that $S(X)$, $S(Y)$ and $S(Z)$ are of type (1), (2) and (3),
respectively. For the first case, it is enough to get any $x$ in a vertex
of the bar of $D$. For the second case, note that
$$
u_1 \supset h^{-r}(u_1) \supset h^{-2r}(u_1) \supset \cdots,
$$
so that the intersection $\bigcap_{n=0}^\infty h^{-nr}(u_1)$ is nonempty;
then take any $y$ in this intersection. Finally, since
$$
w_1 \backslash h(v_s) \supset h^t(w_1 \backslash h(v_s)) \supset
h^{2t}(w_1 \backslash h(v_s)) \supset \cdots,
$$
we may take any $z \in \bigcap_{n=0}^\infty h^{nt}(w_1 \backslash h(v_s))$.
\end{proof}

Given $\alpha \in (\nat \backslash \{1\})^\nat$, consider the product space
$$
\Delta_\alpha = \prod_{i=1}^\infty \integer_{\alpha(i)},
$$
where $\integer_k = \{0,\ldots,k-1\}$ with the discrete topology. Note that
$\Delta_\alpha$ is homeomorphic to the Cantor space. We define an operation
of addition on $\Delta_\alpha$ in the following way: if $(x_1,x_2,\ldots)$
and $(y_1,y_2,\ldots)$ are in $\Delta_\alpha$, then
$$
(x_1,x_2,\ldots) + (y_1,y_2,\ldots) = (z_1,z_2,\ldots),
$$
where $z_1 = x_1 + y_1$ mod $\alpha(1)$ and, in general, $z_i$ is defined
recursively as $z_i = x_i + y_i + \epsilon_{i-1}$ mod $\alpha(i)$ where
$\epsilon_{i-1} = 0$ if $x_{i-1} + y_{i-1} + \epsilon_{i-2} < \alpha(i-1)$
and $\epsilon_{i-1} = 1$ otherwise. If we let $f_\alpha$ be the ``$+1$''
map, that is,
$$
f_\alpha(x_1,x_2,\ldots) = (x_1,x_2,\ldots) + (1,0,0,\ldots),
$$
then $(\Delta_\alpha,f_\alpha)$ is a dynamical system known as a solenoid,
adding machine or odometer. We also define a function $M_\alpha$ from the
set of primes into $\{0,1,2,\ldots,\infty\}$ by
$$
M_\alpha(p) = \sum_{i=1}^\infty n(i),
$$
where $n(i)$ is the largest integer such that $p^{n(i)}$ divides $\alpha(i)$.
The following beautiful characterization of odometers up to topological
conjugacy is due to Buescu and Stewart \cite{BS}:

\smallskip
{\it Let $\alpha, \beta \in (\nat \backslash \{1\})^\nat$. Then
$f_\alpha$ and $f_\beta$ are topologically conjugate if and only if
$M_\alpha = M_\beta$.}

\smallskip
When $M_\alpha(p) = \infty$ for every $p$, $f_\alpha$ is said to be an
universal odometer. It follows from the above-mentioned result
that any two universal odometers are topologically conjugate. 

We shall need the following result from \cite{BK}:

\smallskip
{\it Let $\alpha \in (\nat \backslash \{1\})^\nat$ and
$m_i = \alpha(1) \alpha(2) \cdots \alpha(i)$ for each $i$. Let $f : X \to X$
be a continuous map of a compact topological space $X$. Then $f$ is
topologically conjugate to $f_\alpha$ if and only if {\rm (1)--(3)} hold:
\begin{itemize}
\item [{\rm (1)}] For each positive integer $i$, there is a cover $\p_i$ of
    $X$ consisting of $m_i$ nonempty pairwise disjoint clopen sets which
    are cyclically permuted by~$f$.
\item [{\rm (2)}] For each positive integer $i$, $\p_{i+1}$ refines $\p_i$.
\item [{\rm (3)}] If $W_1 \supset W_2 \supset W_3 \supset \cdots$ is a nested
    sequence with $W_i \in \p_i$ for each $i$, then
    $\bigcap_{i=1}^\infty W_i$ consists of a single point.
\end{itemize}}

\smallskip
Let us also recall that the $\omega$-limit set $\omega(x,f)$ of $f$ at $x$
is the set of all limit points of the sequence $(f^n(x))_{n \in \nat}$.

\begin{thm}\label{LimitSets}
Let $h$ be an element of the comeager conjugacy class of $\homcantor$.
Then, the restriction of $h$ to every $\omega$-limit set $\omega(\sigma,h)$
is topologically conjugate to the universal odometer.
\end{thm}

\begin{proof}
Since $h$ satisfies property (P), we can construct
inductively a sequence $(\p_m)_{m \in \nat}$ of partitions of $\cantor$
and a sequence $(q_m)_{m \in \nat}$ of natural numbers so that the
following properties hold for every $m \in \nat$:
\begin{itemize}
\item $\mesh(\p_m) < 1/m$;
\item $\p_{m+1}$ refines $\p_m$;
\item $q_{m+1}$ is a multiple of $mq_m$;
\item every component of $\gr(h,\p_m)$ is a balanced dumbbell of plate
      weight $q_m!$.
\end{itemize}
Let $\sigma \in \cantor$ and consider the $\omega$-limit set
$\omega(\sigma,h)$. For each $m \in \nat$, $\sigma$ belongs to a vertex of
a certain dumbbell $D_m$ in $\gr(h,\p_m)$. Then, $\omega(\sigma,h)$ must be
contained in one of the loops of $D_m$, for each $m \in \nat$. Thus, it
follows from the above-mentioned result from \cite{BK} that
$h|_{\omega(\sigma,h)} : \omega(\sigma,h) \to \omega(\sigma,h)$
is topologically conjugate to $f_\alpha$, where
$$
\alpha = \big(q_1!,\frac{q_2!}{q_1!},\frac{q_3!}{q_2!},\frac{q_4!}{q_3!},
              \ldots\big).
$$
Since $m!$ divides $\frac{q_{m+1}!}{q_m!}$ for every $m \in \nat$,
$f_\alpha$ is an universal odometer.
\end{proof}

Given a continuous map $f : X \to X$, where $X$ is a metric space,
we shall denote by $P(f)$ (resp.\ $R(f)$, $\Omega(f)$, $CR(f)$) the set of
all periodic points (resp.\ recurrent points, nonwandering points,
chain recurrent points) of $f$ \cite{AH}.

\begin{thm}\label{RecurrentSet}
Let $h$ be an element of the comeager conjugacy class of $\homcantor$. Then, we have that:
\begin{itemize}
\item [{\rm (a)}] $P(h)$ is empty.
\item [{\rm (b)}] $R(h) = \Omega(h) = CR(h)$.
\item [{\rm (c)}] $R(h)$ is a Cantor set with empty interior in $\cantor$.
\end{itemize}
\end{thm}

\begin{proof}
(a): Obvious.

\smallskip
\noindent (b): Since $R(h) \subset \Omega(h) \subset CR(h)$, let us prove that
$CR(h) \subset R(h)$. For this purpose, suppose $\sigma \not\in R(h)$. Then,
there is an $m \in \nat$ such that the set
$$
\{n \in \nat : d(h^n(\sigma),\sigma) < 1/m\}
$$
is finite. Let $\p$ be a partition of $\cantor$ of mesh $< 1/m$ such that
every component of $\gr(h,\p)$ is a dumbbell. Let $D$ be the dumbbell
in $\gr(h,\p)$ which contains a vertex containing $\sigma$ and write
$$
D = \{u_1,\ldots,u_r\} \cup \{v_1,\ldots,v_s\} \cup \{w_1,\ldots,w_t\}
$$
with usual labeling. By our choice of $m$, $\sigma$ must belong either to
the bar $\{v_1,\ldots,v_s\}$ or to the loop $\{u_1,\ldots,u_r\}$ and,
in this last case, its trajectory must leave this loop at some moment.
Both possibilities imply that $\sigma \not\in CR(h)$.

\smallskip
\noindent (c): Let us now prove that $R(h)$ is a Cantor set. Since
$R(h) = \Omega(h)$, which is closed and nonempty, it is enough to show that
$R(h)$ has no isolated point. So, take a point $\tau \in R(h)$. Then
$\tau \in \omega(\tau,h)$. Since $\omega(\tau,h) \subset \Omega(h) = R(h)$
and $\omega(\tau,h)$ is a Cantor set (Theorem~\ref{LimitSets}),
$\tau$ is not an isolated point of $R(h)$.
Finally, suppose $A$ is a nonempty open set of $\cantor$ which is contained
in $R(h)$ and fix $\sigma \in A$. Let $\p$ and $D$ be as in the proof of (b),
where $m$ is chosen so big that the vertex $v$ of $D$ containing $\sigma$
must be contained in $A$. It is easy to see that every point in every set
of the collection
$$
\{h^{-r}(v_1),\ldots,h^{-2}(v_1),h^{-1}(v_1),v_1,\ldots,v_s,
h(v_s),h^2(v_s),\ldots,h^t(v_s)\}
$$
is nonrecurrent. Thus, every vertex of $D$ contains a nonrecurrent point,
contradicting the fact that $v$ is contained in $A$.
\end{proof}

To the best of our knowledge, the fact that the set of  all $h \in \homcantor$
which have no periodic point is comeager in $\homcantor$ 
first appeared in \cite{AHK}.

Recall that a mapping $f$ from a metric space $X$ into itself is said to
be equicontinuous at a point $x \in X$ if for every $\epsilon > 0$ there
exists $\delta > 0$ such that
$$
d(y,x) < \delta \ \ \ \Longrightarrow \ \ \ d(f^n(y),f^n(x)) < \epsilon
  \ \text{ for every } n \geq 0.
$$
Moreover, $f$ is said to be chain continuous at $x$ \cite{akins},~\cite{bernardes}
if for every $\epsilon > 0$ there exists $\delta > 0$ such that for any
choice of points
$$
x_0 \in B(x;\delta), \ x_1 \in B(f(x_0);\delta), \
x_2 \in B(f(x_1);\delta),\ldots,
$$
we have that
$$
d(x_n,f^n(x)) < \epsilon \ \text{ for every } n \geq 0.
$$
Of course, chain continuity is a much stronger property than equicontinuity.

\begin{thm}\label{ChainContinuity}
Let $h$ be an element of the comeager conjugacy class of $\homcantor$.
Then, $h$ is chain continuous at every
nonrecurrent point and so it is chain continuous at every point of a
dense open set, but it is not equicontinuous at each point of an
uncountable set.
\end{thm}

\begin{proof}
Let $\sigma$ be a nonrecurrent point of $h$ and fix $\epsilon > 0$.
Choose $m \in \nat$ such that $1/m < \epsilon$ and the set
$$
\{n \in \nat : d(h^n(\sigma),\sigma) < 1/m\}
$$
is finite. Let $\p$ be a partition of $\cantor$ of mesh $< 1/m$ such that
every component of $\gr(h,\p)$ is a dumbbell and let
$$
D = \{u_1,\ldots,u_r\} \cup \{v_1,\ldots,v_s\} \cup \{w_1,\ldots,w_t\}
$$
be the dumbbell in $\gr(h,\p)$ (with usual labeling) that contains a vertex
containing $\sigma$. We want to find a $\delta > 0$ such that the relations
$\sigma_0 \in B(\sigma;\delta)$, $\sigma_1 \in B(h(\sigma_0);\delta)$,
$\sigma_2 \in B(h(\sigma_1);\delta),\ldots$ imply
$d(\sigma_n,h^n(\sigma)) < \epsilon$ for every $n \geq 0$.
If $\sigma \in v_1 \cup \ldots \cup v_s \cup w_1 \cup \ldots \cup w_t$, then
it is enough to choose $0 < \delta < \epsilon$ smaller than the minimum
distance between two distinct elements of $\p$. If $\sigma$ is in the loop
$\{u_1,\ldots,u_r\}$, our choice of $m$ implies that the trajectory of
$\sigma$ must eventually leave this loop, and so it is clear that we can also
find such a $\delta$ in this case. Thus, we have proved that $h$ is chain
continuous at every point of the set $\cantor \backslash R(h)$, which
is open and dense in view of Theorem~\ref{RecurrentSet}.

Let us now prove the last assertion. Since
$u_1 \supset h^{-r}(u_1) \supset h^{-2r}(u_1) \supset \cdots$,
the closed set
$$
Y = \bigcap_{n=0}^\infty h^{-nr}(u_1)
$$
is nonempty and satisfies $h^r(Y) = Y$. If $Y$ were open, it would follow
from the inclusions $u_1 \backslash Y \supset h^{-r}(u_1 \backslash Y)
\supset h^{-2r}(u_1 \backslash Y) \supset \cdots$ that the intersection
$\bigcap_{n=0}^\infty h^{-nr}(u_1 \backslash Y)$ is nonempty, which is
impossible. Thus, the closed set
$$
Z = \overline{u_1 \backslash Y} \cap Y
$$
is nonempty. Moreover, $h^r(Z) = Z$. If $\sigma \in Z$ then the trajectory
of $\sigma$ remains in the loop $\{u_1,\ldots,u_r\}$ forever (because
$\sigma \in Y$) but as close to $\sigma$ as we want there are points of
$u_1 \backslash Y$ (because $\sigma \in \overline{u_1 \backslash Y}$) and
the trajectories of these points eventually go to the bar of the dumbbell
$D$. This proves that $h$ is not equicontinuous at $\sigma$.
Thus, $h$ is not equicontinuous at each point of the set
$$
Z \cup h(Z) \cup \ldots \cup h^{r-1}(Z).
$$
Since this set is closed and invariant under $h$, it contains the
$\omega$-limit set of each of its elements, and so it is uncountable
in view of Theorem~\ref{LimitSets}.
\end{proof}


\section{The Case of Continuous Maps}

Suppose that $f \in \contcantor$, $\p$ is a partition of $\cantor$ and
$B$ is a component of $\gr(f,\p)$ which is a balloon. Write
$$
B = \{v_1,\ldots,v_s\} \cup \{w_1,\ldots,w_t\},
$$
with usual labeling. We say that $B$ is {\em strict relative to $f$} if
$f(v_i) \subsetneq v_{i+1}$ for every $1 \leq i < s$,
$f(w_j) \subsetneq w_{j+1}$ for every $1 \leq j < t$, and
$f(v_s) \cup f(w_t) \subsetneq w_1$.

Surprisingly enough, we shall now prove that there is a comeager subset
of $\contcantor$ such that any two elements of this set are conjugate
to each other.

\begin{thm}\label{genericthmcont}
Let $S$ be the set of all $f \in \contcantor$ with the following property:

\smallskip
\noindent {\rm (Q)} For every $m \in \nat$, there are a partition $\p$ of
$\cantor$ of mesh $< 1/m$ and a multiple $q \in \nat$ of $m$ such that
every component of $\gr(f,\p)$ is a balloon of type $(q!,q!)$ which is
strict relative to $f$.

\smallskip
\noindent Then, $S$ is a comeager subset of $\contcantor$ such that
any two of its elements are conjugate to each other.
\end{thm}

\begin{proof}
For each $m \in \nat$, let $S_m$ be the set of all $f \in \contcantor$
that satisfies the property contained in (Q) for this particular $m$.
Clearly, each $S_m$ is open in $\contcantor$. In order to prove that
each $S_m$ is also dense $\contcantor$, let us fix $m \in \nat$,
$f \in \contcantor$ and $\epsilon > 0$. It follows from
Theorem~\ref{approxthm}(a) that there are $g \in \contcantor$ with
$\tilde{d}(f,g) < \frac{\epsilon}{2}$, a partition $\p$ of $\cantor$ of mesh
$< \min\{\frac{\epsilon}{2},\frac{1}{m}\}$,
and a multiple $q \in \nat$ of $m$ such that $\gr(g,\p)$ is a digraph whose
components are balloons of type $(q!,q!)$.
If $\psi : \cantor \to \cantor$ maps each $a \in \p$ to a single point
of $a$, then $\psi$ is continuous. Moreover, $\psi \circ g \sim_\p g$,
which implies that $\gr(\psi \circ g,\p) = \gr(g,\p)$ and
$\tilde{d}(\psi \circ g,g) \leq \mesh(\p) < \frac{\epsilon}{2}$
(and so $\tilde{d}(\psi \circ g,f) < \epsilon$). Since $\psi \circ g$
has finite range, each component (balloon) of $\gr(\psi \circ g,\p)$ is strict
relative to $\psi \circ g$. Hence, $\psi \circ g \in S_m$, proving that
$S_m$ is dense in $\contcantor$. Thus, $S = \bigcap S_m$ is a comeager
subset of $\contcantor$.

Let $f \in \contcantor$. We say that a partition $\p$ of $\cantor$ is
{\em $f$-admissible} if there is a $k \in \nat$ such that every component
of $\gr(f,\p)$ is a balloon of type $(k,k)$ which is strict relative to $f$.
In this case, we denote this number $k$ by $b(f,\p)$. If $f \in S$ then
there are $f$-admissible partitions $\p$ such that $\mesh(\p)$ is as small
as we want and $b(f,\p)$ is a multiple of any positive integer we want.

Suppose $\p$ and $\p'$ are $f$-admissible partitions. If $\mesh(\p')$ is
sufficiently small, then $\p'$ is necessarily a refinement of $\p$.
Assume that this is the case. Then, each component $B'$ of $\gr(f,\p')$
must be {\em contained} in some component $B$ of $\gr(f,\p)$, in the sense
that the union of all vertices of $B'$ is contained in the union of all
vertices of $B$. Moreover, $b(f,\p')$ is necessarily a multiple of $b(f,\p)$.
Let $B$ be a component of $\gr(f,\p)$. We say that a component $B'$ of
$\gr(f,\p')$ is a {\em subballoon of $B$ of type $u$} if the initial vertex
of $B'$ is contained in the vertex $u$ of $B$. With this definition,
$B$ can be thought of as the union of its subballoons relative to $\p'$.
Since the balloon $B$ is strict relative to $f$, there must exist
subballoons of $B$ of every type $u \in B$ provided $\mesh(\p')$
is sufficiently small. More precisely, we can make the number of
subballoons of $B$ of each type $u$ as large as we want by choosing
$\p'$ with $\mesh(\p')$ small enough.

If $f \in S$ and $g \in \contcantor$ is conjugate to $f$, then it is easy
to verify that $g \in S$. Take $f, g \in S$ and let us prove that
$f$ and $g$ are conjugates. It is enough to construct sequences
$(\p_n)$, $(\q_n)$ and $(\nu_n)$ with the properties described in part (ii)
of Theorem~\ref{ConjugacyRelation}.

We begin by taking a $g$-admissible partition $\q_1$ with $\mesh(\q_1) < 1$.
Then, we take an $f$-admissible partition $\p_1$ such that $\mesh(\p_1) < 1$,
$b(f,\p_1)$ is a multiple of $b(g,\q_1)$ and the set $X$ of all components of
$\gr(f,\p_1)$ has cardinality greather than or equal to that of the
set $Y$ of all components of $\gr(g,\q_1)$. Finally, we choose a surjection
$\phi : X \to Y$ and, for each $B \in X$, we define $\nu_1$ on $B$ as the
unique surjection from $B$ onto $\phi(B)$ that maps the initial vertex of $B$
to the initial vertex of $\phi(B)$ and satisfies the relation
$$
\edge{ab} \in \gr(f,\p_1) \ \Longrightarrow \
\edge{\nu_1(a)\nu_1(b)} \in \gr(g,\q_1) \ \ \ \ \ (a, b \in B).
$$
In this way, we obtain a surjective graph map
$\nu_1 : \gr(f,\p_1) \to \gr(g,\q_1)$.

Now, we take an $f$-admissible partition $\p_2$ such that $\p_2$ refines
$\p_1$, $\mesh(\p_2) < 1/2$ and every component of $\gr(f,\p_1)$ has
subballoons of every type relative to $\p_2$. Then, we take a $g$-admissible
partition $\q_2$ such that $\q_2$ refines $\q_1$, $\mesh(\q_2) < 1/2$,
$b(g,\q_2)$ is a multiple of $b(f,\p_2)$ and every component of $\gr(g,\q_1)$
has subballoons of every type relative to $\q_2$. Let us fix a component $B$
of $\gr(g,\q_1)$ and let $\{B_1,\ldots,B_r\}$ be the set of all components
$B_k$ of $\gr(f,\p_1)$ such that $\nu_1(B_k) = B$. For each $u \in B$,
let $X_u$ be the set of all subballoons of $B$ (relative to $\q_2$)
of type $u$. Moreover, for each $1 \leq k \leq r$, let $Y_{k,u}$ be the
set of all subballoons of $B_k$ (relative to $\p_2$) of type $v$ for some
$v \in \nu_1^{-1}(\{u\})$. We may assume that $\q_2$ was chosen so that
$$
{\rm Card}\, X_u \geq {\rm Card}(Y_{1,u} \cup \ldots \cup Y_{r,u})
\ \ \text{ for every } u \in B.
$$
Hence, we may choose a surjection
$\phi_u : X_u \to Y_{1,u} \cup \ldots \cup Y_{r,u}$ ($u \in B$).
Finally, for each $u \in B$ and each $B' \in X_u$,
we define $\nu_2$ on $B'$ as the unique surjection from $B'$ onto $\phi_u(B')$
that maps the initial vertex of $B'$ to the initial vertex of $\phi_u(B')$
and satisfies the relation
$$
\edge{ab} \in \gr(g,\q_2) \ \Longrightarrow \
\edge{\nu_2(a)\nu_2(b)} \in \gr(f,\p_2) \ \ \ \ \ (a, b \in B').
$$
We claim that
$$
(\ast) \ \ \ \ \ j_1(a) = \nu_1(i_1(\nu_2(a))) \ \ \text{ for every } a \in B',
$$
where $i_1 : \p_2 \to \p_1$ and $j_1 : \q_2 \to \q_1$ are the refinement maps.
In fact, let us first consider the initial vertex $c$ of $B'$.
Let $k \in \{1,\ldots,r\}$ be such that $\phi_u(B') \in Y_{k,u}$.
Since $j_1(c) = u$ (because $B' \in X_u$) and
$i_1(\nu_2(c)) \in \nu_1^{-1}(\{u\})$ (because $\nu_2(c)$ is the initial
vertex of $\phi_u(B')$ and $\phi_u(B') \in Y_{k,u}$), it follows that $c$
satisfies the equality in $(\ast)$. Now, let us assume that a certain
vertex $a$ of $B'$ satisfies the equality in $(\ast)$. Let $b$ be the
unique vertex of $B'$ such that $\edge{ab} \in B'$. Since
$$
\edge{ab} \in B' \ \Longrightarrow \ \edge{j_1(a)j_1(b)} \in B,
$$
\begin{align*}
\edge{ab} \in B' \ & \Longrightarrow \ \edge{\nu_2(a)\nu_2(b)} \in \phi_u(B')\\
  & \Longrightarrow \ \edge{i_1(\nu_2(a))i_1(\nu_2(b))} \in B_k\\
  & \Longrightarrow \ \edge{\nu_1(i_1(\nu_2(a)))\nu_1(i_1(\nu_2(b)))} \in B,
\end{align*}
and we are assuming that $j_1(a) = \nu_1(i_1(\nu_2(a)))$, it follows that
$b$ also satisfies the equality in $(\ast)$. By induction,
we see that $(\ast)$ holds.
Thus, by defining $\nu_2$ in this way for each component $B$ of $\gr(g,\q_1)$,
we obtain a surjective graph map $\nu_2 : \gr(g,\q_2) \to \gr(f,\p_2)$
such that $j_1 = \nu_1 \circ i_1 \circ \nu_2$.

Now, we apply exactly the same procedure to construct $\p_3$, $\q_3$
and $\nu_3$ (but with $\p_3$ and $\q_3$ in place of $\q_2$ and $\p_2$,
respectively), and so on. This completes the proof.
\end{proof}

Let us now establish some applications to dynamics.

It was proved in \cite{dd} that the set of elements of $\contcantor$
which have zero topological entropy and no periodic points is comeager in
$\contcantor$. Hence, an element chosen at ``random'' from $\contcantor$ is
not chaotic in the sense of entropy nor in the sense of Devaney.
We show something much stronger below. 

\begin{thm}\label{LiYorkecont}
There is a comeager subset of $\contcantor$, no element of which
has a Li-Yorke pair.
\end{thm}

\begin{proof}
Let $f \in \contcantor$ satisfy property (Q) of Theorem~\ref{genericthmcont}.
Suppose that $\sigma,\tau \in \cantor$ satisfy
$$
\liminf_{n \to \infty} d(f^n(\sigma),f^n(\tau)) = 0.
$$
Fix $\epsilon > 0$ and choose $m \in \nat$ such that $1/m < \epsilon$.
Then, there is a partition $\p$ of $\cantor$ of mesh $< 1/m$ such that
every component of $\gr(f,\p)$ is a balloon. Moreover, there must exists an
$n_0 \in \nat$ such that both $f^{n_0}(\sigma)$ and $f^{n_0}(\tau)$ lie in
the same vertex $a$ of $\gr(f,\p)$. Let $B$ be the component of $\gr(f,\p)$
that contains $a$. Since $B$ is a balloon, $f$ maps each vertex of $B$
into a vertex of $B$. Hence, both $f^n(\sigma)$ and $f^n(\tau)$ lie in
the same vertex of $B$, so that $d(f^n(\sigma),f^n(\tau)) < \epsilon$,
for each $n \geq n_0$. This proves that
$$
\lim_{n \to \infty} d(f^n(\sigma),f^n(\tau)) = 0,
$$
and so $(\sigma,\tau)$ is not a Li-Yorke pair for $f$.
\end{proof}

In contrast to the case of homeomorphisms (Theorem~\ref{ChainContinuity}),
we have the following

\begin{thm}\label{ChainContinuityCont}
The set of maps $f$ of $\contcantor$ such that $f$ is chain continuous
at every point is comeager in $\contcantor$.
\end{thm}

\begin{proof}
Let $f \in \contcantor$ satisfy property (Q).
Fix $\epsilon > 0$ and choose $m \in \nat$ such that $1/m < \epsilon$.
Then, there is a partition $\p$ of $\cantor$ of mesh $< 1/m$ such that
every component of $\gr(f,\p)$ is a balloon. Let $\delta$ be the minimum
distance between two distinct elements of $\p$. Given $\sigma \in \cantor$,
let $B$ be the component of $\gr(f,\p)$ which contains a vertex
containing $\sigma$. Since $B$ is a balloon, $f$ maps each vertex of $B$
into a vertex of $B$. Hence, if $\sigma_0 \in B(\sigma;\delta)$,
$\sigma_1 \in B(f(\sigma_0);\delta)$, $\sigma_2 \in B(f(\sigma_1);\delta),
\ldots$, then both $\sigma_n$ and $f^n(\sigma)$ lie in the same vertex of $B$,
so that $d(\sigma_n,f^n(\sigma)) < \epsilon$, for every $n \geq 0$.
\end{proof}

It was proved in \cite{dds} that there is a comeager subset of $\contcantor$
such that each element $f$ in this set has property that the restriction of
$f$ to the $\omega$-limit set $\omega(\sigma,f)$ is
topologically conjugate to the universal odometer for a comeager set of
$\sigma \in \cantor$. The next result tell us that this actually holds for
{\em every} point $\sigma \in \cantor$.

\begin{thm}\label{LimitSetsCont}
There is a comeager subset of $\contcantor$ such that each $f$ in
this set has the following property:
The restriction of $f$  to every
$\omega$-limit set $\omega(\sigma,f)$ is topologically conjugate to the
universal odometer.
\end{thm}

\begin{proof}
The proof of this result is similar to that of Theorem~\ref{LimitSets}
and so we omit it.
\end{proof}

\begin{thm}\label{RecurrentSetCont}
The set of all $f \in \contcantor$ which satisfies the following properties is
comeager in $\contcantor$:
\begin{itemize}
\item [{\rm (a)}] $P(f)$ is empty.
\item [{\rm (b)}] $R(f) = \Omega(f) = CR(f)$.
\item [{\rm (c)}] $R(f)$ is a Cantor set with empty interior in $f(\cantor)$.
\end{itemize}
\end{thm}

\begin{proof}
Let $f \in \contcantor$ satisfy property (Q).

\smallskip
\noindent (a): Obvious.

\smallskip
\noindent (b): The fact that $f$ is chain continuous at every point
(Theorem~\ref{ChainContinuityCont}) clearly implies that
$CR(f) \subset R(f)$, and so (b) holds.

\smallskip
\noindent (c): If $\tau \in R(f)$ then $\tau \in \omega(\tau,f) \subset R(f)$,
which implies that $\tau$ is not an isolated point of $R(f)$ since
$\omega(\tau,f)$ is a Cantor set (Theorem~\ref{LimitSetsCont}).
In view of (b), we conclude that $R(f)$ is a Cantor set. Finally,
suppose that $U$ is a nonempty open set of $f(\cantor)$ which is
contained in $R(f)$. Fix $\sigma \in U$ and let $V$ be an open set of
$\cantor$ such that $U = V \cap f(\cantor)$. Let $\p$ be a partition
of $\cantor$ such that every component of $\gr(f,\p)$ is a balloon and
$\mesh(\p)$ is so small that the vertex $v$ of $\gr(f,\p)$ containing
$\sigma$ must be contained in $V$. Let
$$
B = \{v_1,\ldots,v_s\} \cup \{w_1,\ldots,w_t\}
$$
be the component (balloon) of $\gr(f,\p)$ which contains the vertex $v$.
Since $\sigma \in R(f)$, $v$ must be one of the vertices $w_1,\ldots,w_t$;
say $v = w_j$. Then
$$
f^j(v_s) \subset w_j = v \subset V \ \ \ \text{ and } \ \ \
f^j(v_s) \cap R(f) = \emptyset,
$$
which implies that $f^j(v_s) \subset U \backslash R(f) = \emptyset$,
a contradiction.
\end{proof}

We remark that it was established in \cite{dd} that
the set of $f\in \contcantor$ without periodic point forms a comeager
subset of $\contcantor$.


\vspace{.1in}
\centerline{\sc Acknowledgement}

\vspace{.1in}
The first author was partially supported by CAPES:
Bolsista - Proc. n$^{\rm o}$ BEX 4012/11-9.

The second author thanks the Mathematics Institute of Federal University
of Rio de Janeiro for its support and hospitality during his one-month
visit in 2011.

The authors are grateful to the anonymous referee for making several valuable changes
which significantly improved the exposition of the paper.


\vspace{.2in}

\end{document}